\documentclass[10pt]{amsart}
\usepackage{amssymb}
\usepackage{amsmath}
\usepackage{amsthm}
\usepackage{latexsym}
\usepackage{graphicx}
\usepackage{hyperref}

\newcommand{\Z}{{\mathbb Z}}
\newcommand{\R}{{\mathbb R}}

\newcommand{\T}{{\mathbb T}}

\newtheorem{lemma}{Lemma}[section]
\newtheorem{theorem}[lemma]{Theorem}

\newtheorem{proposition}[lemma]{Proposition}
\newtheorem{corollary}[lemma]{Corollary}
\newtheorem{definition}[lemma]{Definition}

\newcommand{\nn}{\nonumber}
\newcommand{\be}{\begin{equation}}
\newcommand{\ee}{\end{equation}}

\newcommand{\ol}{\overline}
\newcommand{\ti}{\tilde}

\newcommand{\spr}[2]{\left\langle #1 , #2 \right\rangle}

\newcommand{\E}{\mathrm{e}}
\newcommand{\I}{\mathrm{i}}

\DeclareMathOperator{\dist}{dist}

\newcommand{\eps}{\varepsilon}

\newcommand{\lam}{\lambda}

\numberwithin{equation}{section}

\begin{document}

\title{On the spectrum of skew-shift Schr\"odinger operators}

\author[H.\ Kr\"uger]{Helge Kr\"uger}
\address{Erwin Schr\"odinger Institute, Boltzmanngasse 9, A-1090 Vienna, Austria}
\email{\href{mailto:helge.krueger@gmail.com}{helge.krueger@gmail.com}}
\urladdr{\href{http://math.rice.edu/~hk7/}{http://math.rice.edu/\~{}hk7/}}

\thanks{H.\ K.\ was supported an Erwin Schr\"odinger junior research fellowship.}

\date{\today}


\begin{abstract}
 I prove that the spectrum of a skew-shift Schr\"odinger operator
 contains larges interval in the semi-classical regime. In the 
 semi-classical limit, these intervals approach the range of the potential.
\end{abstract}

\maketitle

%
%
%

\section{Introduction}

In this paper, I consider the discrete 
Schr\"odinger operator $H_{h} = h \Delta + V$
on $\ell^2(\Z)$ with potential
\be\label{eq:potnsq}
 V(n) = f(\alpha n^2),
\ee
where $\alpha$ is a Diophantine number and $f$ a one-periodic
real-analytic function. In the semi-classical regime, that is
for $h > 0$ sufficiently small, I will show that the spectrum
of the operator $H_{h}$ contains large intervals (see Theorem~\ref{thm:int2}),
which approach the range of $f$ as $h \to 0$.
I will refer to the Schr\"odinger operator with
potential given by \eqref{eq:potnsq} as the {\em skew-shift
Schr\"odinger operator}, since its potential can be generated
by evaluating a sampling function along the orbit
of the skew-shift:
\be\label{eq:skewshift}\begin{split}
 T_{\alpha}: \T^2 & \to \T^2,\\
 T_{\alpha}(x,y) &= (x + 2 \alpha, x + y) \pmod{1},
\end{split}\ee
where $\T = \R/\Z$ is the unit circle. It is expected
that for all $h > 0$ the spectrum of the operator $H_h$
is an interval, see the end of Chapter~15 in Bourgain's book
\cite{bbook} and the end of this introduction.
Let me mention at this point that understanding the model
above is of physical relevance due to its relation to
the kicked-rotor problem, see Chapter~16 in \cite{bbook}.

My result extends the result of Bourgain from \cite{b2}, \cite{b3} that
the spectrum has positive measure. As this result, the proof
proceeds by analyzing a parametrization of the eigenvalues of
finite restrictions. One might wonder, why it is hard to prove
that the spectrum consists of intervals, since it is true
for simple examples like the free Laplacian $\Delta$ or periodic
Schr\"odinger operators. Maybe, the most obvious obstructions
are the results of Avila, Bochi, and Damanik \cite{abd1}, \cite{abd2},
which show that generic potentials have Cantor spectrum. Before
discussing results on Cantor spectrum further, I will comment
on previous results showing interval spectrum.

Since the spectrum of random Schr\"odinger operators is the
union of periodic spectra, these consists of intervals. See
\cite{kirsch}, \cite{kcorr}, \cite{nist} for implementations
of this. Also my results concerning the potential $V(n) = f(n^{\rho})$
for $\rho > 0$ not an integer from \cite{krho} are an implementation
of this fact, since they boil down to showing that there are
arbitrarily long stretches of $n$, where $V(n) \approx f(x)$ for any $x\in\T$.
It is also possible to construct limit-periodic examples with
spectrum containing intervals see \cite{gk} and \cite{poesch}.
Most relevant to us is the work of Chulaevsky and Sinai \cite{cs},
where they show that the spectrum of a two-periodic quasi-periodic
Schr\"odinger operators is a single interval for a set of
frequencies $\mathcal{A}$, which approaches all possible frequencies
as $h\to 0$. However, Bourgain has shown in \cite{b3} that these
models also exhibit gaps for arbitrary $h > 0$ but extreme frequencies.

\quad

My result that the spectrum of the skew-shift Schr\"odinger operator
contains large intervals distinguishes these from
one-frequency quasi-periodic
Schr\"odinger operators, where the potential is given by
\be\label{eq:potqp}
 V_{\mathrm{QP}}(n) = f(\alpha n).
\ee
In fact for these Goldstein and Schlag \cite{gs4} have shown 
that the spectrum is a Cantor set for almost every frequency
$\alpha$ in the regime of $h > 0$ small.
At this point, let me also point out that
Avila and Jitomirskaya have solved the so called
{\em Ten Martini Problem} \cite{aj1}, \cite{aj2},
which asked to show that for $f(x) = 2 \cos(2\pi x)$
the operator with potential given by \eqref{eq:potqp}
has Cantor spectrum for any irrational $\alpha$ and
any $h> 0$.

As already mentioned, Avila, Bochi, and Damanik \cite{abd1}, \cite{abd2} have shown that
for generic continuous sampling function and a large class
of base transformations, one has Cantor spectrum.
This result applies in particular to the skew-shift potential
\eqref{eq:potnsq} with any irrational frequency $\alpha$.
I will further comment on the results from \cite{abd1}
when discussing optimality of my results.

If the sampling functions only takes finitely many values,
it is known some generality, that the spectrum has zero
Lebesgue measure, and thus is a Cantor set. This follows
from the work of Damanik and Lenz \cite{dl06a}, \cite{dl06b}.

\quad

Before coming to the technical discussion of the results,
let me come to an aspect of the proof, I find surprising.
If one iterates the the skew-shift defined in \eqref{eq:skewshift},
one finds
\be
 (T_{\alpha})^n(x,y) = (x + 2 n \alpha, y + n x + n (n - 1) \alpha) \pmod{1}.
\ee
One might expect that the relevant part of the dynamics
is encoded in the $\alpha n^2$ term, which is expected to
behave like random variables \cite{rsz01}. However, we will not make
use of this, but exploit the $n x$ term, to obtain independence
of events, which are far enough apart. This is possible,
since $x$ enters the problem as a {\em fast variable},
see Section~\ref{sec:elifast} for the implementation of this
fact.

However, the equidistribution properties of the sequence
$\alpha n^2$ enters the proof of the large deviation estimates,
see \cite{b2002}, \cite{bbook}, and \cite{bgs}.

\quad

Last, let me mention that it is easy to produce overwhelming
numerical evidence for that the spectrum of $H_h = h \Delta + V$
is an interval for all $h > 0$. I will discuss this
in Appendix~\ref{sec:numeric}.

%
%
%

\section{Statement of the results}

I will now make the statement of my result precise. First,
let me specify that I will assume the frequency $\alpha$
satisfies for some $c > 0$ the Diophantine condition
\be\label{eq:conddiop}
 \| q \alpha \| \geq \frac{c}{q^2},
\ee
for all integers $q \geq 1$, where $\|x\| = \dist(x,\Z)$.
Next, we will need the following result, which can be
proven by the methods of Bourgain, Goldstein, and Schlag
\cite{bgs} or the ones of Bourgain \cite{b2002}.

\begin{theorem}[\cite{b2002}, \cite{bgs}]
 \label{thm:intldt}
 There exists $h_0 = h_0(c,f) > 0$ such that
 for $0 < h < h_0$, we have that large deviation
 estimates for the Green's function hold.
\end{theorem}

I will give a precise meaning to the large deviation
estimates in Theorem~\ref{thm:ldtgreen}. The requirement
that the large deviation estimates hold is the first
smallness condition on $h$, I impose.

The second smallness condition on $h > 0$ is required to verify 
the initial condition of the inductive scheme.
Given $\delta > 0$ define a set of energies by
\be
 \mathcal{E}_{\delta} = \{E:\quad \exists\ x:
  \quad f(x) = E\text{ and }|f'(x)| \geq \delta\}.
\ee
This set is clearly an union of intervals. I am now
ready to state the main result of the paper

\begin{theorem}\label{thm:int2}
 There exists $h_1 = h_1(\delta, f) > 0$ such that
 for $0 < h < h_1$ we have
 \be
  \mathcal{E}_{\delta} \subseteq \sigma(H_{h})
 \ee
 if the large deviation estimates for the Green's function hold.
\end{theorem}

Here, $\sigma(H_{h})$ denotes the spectrum of the
operator $H_{h} = h \Delta + V$, where $V$
was given by \eqref{eq:potnsq}. The smallness
condition in this theorem does not depend on
the Diophantine condition \eqref{eq:conddiop},
but the condition that the large deviation
estimates hold, imposes such a condition through
Theorem~\ref{thm:intldt}. In Appendix~\ref{sec:topbotspec},
I will demonstrate that
\be \begin{split}
 \min(f) - 2h &\leq \min(\sigma(H_{h})) \leq \min(f) - h,\\
 \max(f) + h  &\leq \max(\sigma(H_{h})) \leq \max(f) + 2h,
\end{split}\ee
which shows that Theorem~\ref{thm:int2} covers most of the
spectrum.

I will prove a more precise result
than Theorem~\ref{thm:int2}. Recall that $f: \T \to \R$ is
a non-constant real-analytic function. 
For $x,y \in \T$ and $\alpha$ Diophantine, we introduce the
family of potentials
\be\label{eq:potskew}
 V_{x,y}(n) = f( (T_{\alpha} ^n (x,y))_2 ) 
  = f(y + n x + n (n-1) \alpha).
\ee
Let now $h > 0$, then we introduce the family of
skew-shift Schr\"odinger operators by
\be\label{eq:defHhxy}\begin{split}
 H_{h,x,y}: \ell^2(\Z)&\to\ell^2(\Z),\\
 H_{h,x,y} u(n) &= h \Big( u(n+1) + u(n-1)\Big) + V_{x,y}(n) u(n).
\end{split}\ee
In short notation, we have $H_{h,x,y} = h \Delta + V_{x,y}$.
We also note that $H_h = H_{h,\alpha,0}$.
The more precise result is

\begin{theorem}\label{thm:int1}
 There exists $h_1 = h_1(\delta, f) > 0$. Let
 $0 < h < h_1$ such that the large deviation
 estimates for the Green's function hold.
 Then for $E \in \mathcal{E}_{\delta}$, there exist
 $(x,y) \in \T^2$ such that $E$ is an eigenvalue of
 $H_{h,x,y}$.
\end{theorem}

By minimality of the skew-shift, we have that
$\sigma(H_{h,x,y}) = \sigma(H_{h,\ti{x},\ti{y}})$
for any $x,y,\ti{x},\ti{y} \in \T^2$.
Using this, it is easy to show that
Theorem~\ref{thm:int1} implies Theorem~\ref{thm:int2}.
Minimality of the skew-shift even implies that
that the set of $(x,y) \in \T^2$ such that
$E$ is an eigenvalue of $H_{h,x,y}$ is dense.
However, it follows from the general theory of 
ergodic Schr\"odinger operators, that this set has
zero measure.

It should be mentioned here that it is an open question
whether the operator $H_{h, x,y}$ has pure point
spectrum for almost every $(x,y)$ or not. The results
of Bourgain, Goldstein, and Schlag only imply pure point
spectrum for almost every frequency $\alpha$. So
in some sense, Theorem~\ref{thm:int1} exhibiting at
least one eigenvalue for some $x,y$ is a step
towards proving that $H_{h,x,y}$ has pure point spectrum
for almost every $(x,y) \in \T^2$.


\subsection{Optimality of the results}

It might seem that the choice of potential in \eqref{eq:potskew}
is somewhat arbitrary, since we assume that the function
only depends on the second coordinate. However, when one
considers potentials of the more general form
\be
 V_{x,y}(n) = g(T^n (x,y))
\ee
for a real analytic function $g:\T^2 \to \R$, one faces obstruction
to the spectrum being an interval.

Consider $g$ of the form
\be
 g(x,y) = 2 \cos(2\pi x) + \kappa \cos(2\pi y)
\ee
for some small $\kappa > 0$. For $\kappa = 0$, the
operator $H_h$ will just be the Almost--Mathieu operator,
which is known to have Cantor spectrum, in particular
it has at least one gap of size at least $\eta$ for 
$\eta > 0$ sufficiently small. Hence, if $\kappa < \frac{\eta}{2}$
also the operator with skew-shift potential depending
non-trivially on the second coordinate has at least
one gap in its spectrum.

For fixed $h > 0$ and any continuous function $f: \T^2 \to\R$,
Avila, Bochi, and Damanik have shown in \cite{abd1} that
there exists a continuous function $f_1:\T^2 \to \R$
such that $\|f-f_1\|_{L^{\infty}(\T^2)}$ is arbitrarily small
and the Schr\"odinger operator $H_{h,1} = h \Delta + V_1$
with
\be
 V_1(n) = f_1(T^n (0,0))
\ee
has Cantor spectrum. In particular, the spectrum contains
a gap of size $2 \eta$, that is there is some $E_0$ such
that
\be
 \sigma(h \Delta + V_1) \cap [E_0 - \eta, E_0 + \eta]
  = \{E_0 - \eta, E_0 + \eta\}.
\ee
It is classical that there exists now an analytic function
$f_2$ such that $\|f_1 - f_2\|_{L^{\infty}(\T^2)} \leq \frac{\eta}{2}$.
Then standard perturbation theory shows that
\be
 \sigma(h \Delta + V_2) \cap (E_0 - \frac{\eta}{2}, E_0 + \frac{\eta}{2}) \neq \emptyset,
\ee
where $V_2 (n) = f_2(T^n(0,0))$. Hence, this operator has
a gap in the spectrum.

However, an inspection of the argument of my proof, shows that
my result is stable under perturbing the sampling function
$f$ in the $C^1(\T^2)$ topology as long as the large deviation
estimates continue to hold. For this it is necessary, that the 
domain of analyticity of $f$ stays the same.


\subsection{Discussion of the proof}

I will now try to explain the main ideas in this paper.
Let me begin by pointing out that checking the initial
condition of the multi-scale scheme in Section~\ref{sec:initcond},
is done by a computation, I used in \cite{kgaps}, to show that
all gaps $[E_-, E_+]$ of $\sigma(H_{h})$ must satisfy
$E_+ - E_- = O(h^2)$ as $h\to 0$.

As mentioned above the proof
proceeds by a multi-scale scheme. This scheme bears
some similarities to the one used by Bourgain in \cite{b2}, \cite{b3}
to prove that the measure of the spectrum of quasi-periodic
Schr\"odinger operators in the localization regime
is positive. A key difference is that the arguments of 
this paper use analytic perturbation theory to show that
a fixed number $E_0$ belongs to the spectrum, see
Section~\ref{sec:evperturb}. This is necessary, since 
extending an eigenvalue from scale to scale slightly
perturbs it.

Maybe the key insight
was that since the $n$th iterate of the skew-shift is
\be
 T_{\alpha}^{n} (x,y) = (x + 2 n \alpha, y + n x + n (n-1) \alpha)
  \pmod{1},
\ee
one has that $x$ enters the problem as a {\em fast variable},
since it gets multiplied by $n$ in the second coordinate.
This realization will allow us to prove an elimination
of $x$ argument in Section~\ref{sec:elifast}, which is used
to eliminate {\em double resonances}.
At this point let me also mention that the argument
of Section~\ref{sec:elifast} is an adaptation of the
frequency elimination argument of Bourgain and Goldstein
from \cite{bg}.

As a further point of interest, let me point out that of
the arguments to prove the presence of gaps in the spectrum,
my argument is most related to the one of Goldstein and
Schlag \cite{gs4}, see also \cite{gsfest} for a non-technical
discussion. My argument shows that there are always
simple resonances, so that I can eliminate double resonances,
whereas Goldstein and Schlag show that there are certain
simple resonances, that must also be double resonances, and not
triple. The formation of double resonances then implies
that gaps must open.

I furthermore wish to point out, what the methods of
this paper would yield for the one-frequency quasi-periodic
model, so consider the potential
\be
 V_{x,\alpha}^{\mathrm{QP}} (n) = f(n \alpha + x)
\ee
and the associated Schr\"odinger operator 
$H_{h,x,\alpha}^{\mathrm{QP}} = h\Delta + V_{x,\alpha}^{\mathrm{QP}}$.
Now, as in \cite{bg}, $\alpha$ will play the role of
a fast-variable. Translating the statement of
Theorem~\ref{thm:int1}, one obtains that for every $E \in \mathcal{E}_{\delta}$,
there exists a set $\mathcal{A} = \mathcal{A}(E, h)$ such that
\begin{enumerate}
 \item $|\mathcal{A}| \to 1$ as $h \to 0$.
 \item For $\alpha \in \mathcal{A}$, there exists $x \in \T$
  such that $E$ is an eigenvalue of $H_{h,x,\alpha}^{\mathrm{QP}}$.
\end{enumerate}
It is clear that this is compatible with quasi-periodic
Schr\"odinger operators having gaps in their spectrum.

\subsection{A non-technical description of the proof}

\begin{figure}[ht]
 \includegraphics[width=\textwidth]{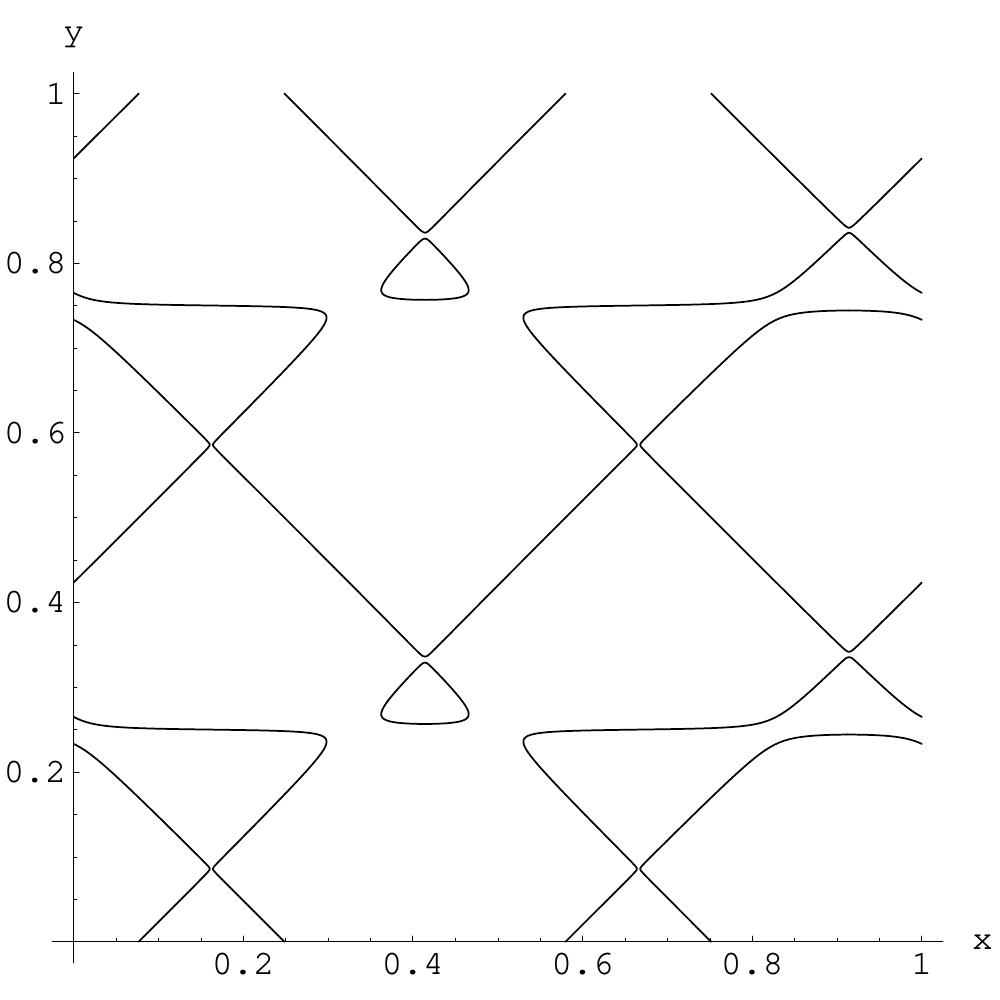}
 \caption{The $(x,y)$ such that $0 \in \sigma(H^{[-1,1]}_{0.1, x,y})$.}
 \label{fig:1}
\end{figure}


Having now explained the main ideas behind the proof,
let me explain some of the details. Denote by $H^{[-N,N]}_{h,x,y}$
the restriction of $H_{h,x,y}$ to $\ell^2(\{-N,\dots,N\})$.
Given $E_0 \in \mathcal{E}_{\delta}$, we will construct
inductively a sequence $N_j$ and curves
$\xi_j: \T \to \T$ such that for a positive measure set
$\mathfrak{X}_j$, we have
\be
 E_0\text{ is an eigenvalue of } H_{h,x, \xi_j(x)}^{[-N_j,N_j]}
\ee
whenever $x \in \mathfrak{X}_j$. The main problem with
this approach is to pass from scale to scale. In order
to discuss some aspects of this problem, I have included
Figure~\ref{fig:1}, which shows the set of $(x,y)$ such 
that
\be
 0\text{ is an eigenvalue of } H_{0.1, x, y}^{[-1,1]}
\ee
for the potential
\be
 V_{x,y}(n) = \cos(2\pi (\sqrt{2} n(n-1) + n x + y)).
\ee
One should notice in this figure that there are parts
of two almost straight segments around $y = 0.25$ and $y = 0.75$.
These correspond to the fact that
\be
 V_{x,y}(0) = 0,\quad y \in \{0.25, 0.75\}.
\ee
The interruptions in these lines can be identified with the
set of $x$ such that there exists $n \in \{-1,0,1\}$
such that
\be
 V_{x,0.25}(n) = 0 \quad\text{respectively}\quad
 V_{x,0.75}(n) = 0.
\ee
This is illustrated in Figure~\ref{fig:1dash}, where these
lines are shown dashed.
Making these assertions precise is the content of the first
step in the proof of the initial condition given in
Section~\ref{sec:initcond}.

\begin{figure}[ht]
 \includegraphics[width=\textwidth]{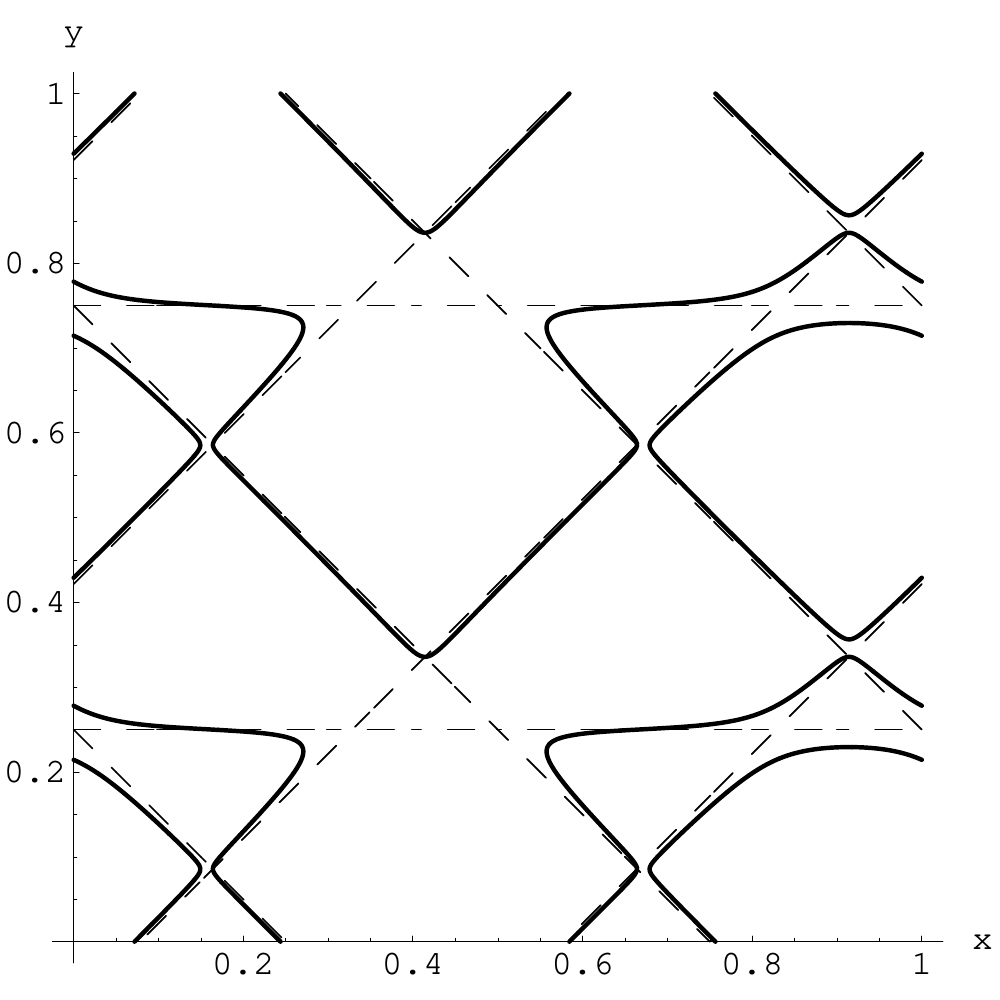}
 \caption{The $(x,y)$ such that $0 \in \sigma(H^{[-1,1]}_{0.2, x,y})$ in thick.
  The $(x,y)$ such that $V_{x,y}(n) = 0$ for $n = -1,0,1$ are dashed.}
 \label{fig:1dash}
\end{figure}

I have also included Figure~\ref{fig:2}, which shows the same
situation as Figure~\ref{fig:1} except for $N = 2$ instead
of $N = 1$.

\begin{figure}[ht]
 \includegraphics[width=\textwidth]{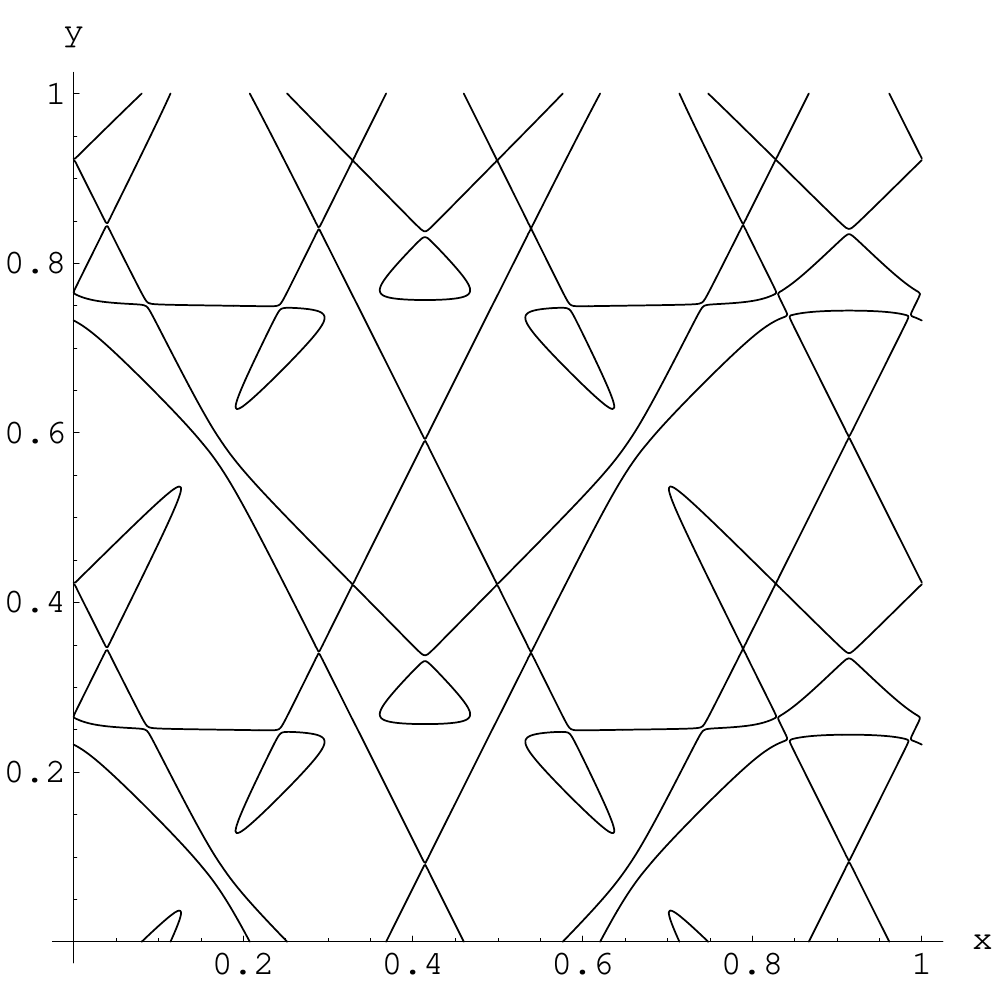}
 \caption{The $(x,y)$ such that $0 \in \sigma(H^{[-2,2]}_{0.1, x,y})$.}
 \label{fig:2}
\end{figure}

The explanations so far explain, why there are eigenvalues
close to the line $y = 0.25$ for many $x$. Let me now
mention that using analytic perturbation theory, one can
construct a function $\xi_1$ such that for many $x$,
we have
\be
 0\text{ is in the spectrum of } H_{h, x, \xi_1(x)}^{[-N_1, N_1]}
\ee
for some $N_1 \geq 1$. This function $\xi_1$ will satisfy
that $|\xi_1'(x)|$ and $|\xi_1(x) - 0.25|$ are both small.

In order to pass to the scale $N_2$, we will exploit that
$x$ is a fast variable and the large deviation estimates
for the Green's functions. These will allow us to show that
for
\be
 -N_2 \leq n \leq -\frac{N_1}{10},\quad
 \frac{N_1}{10} \leq n \leq N_2,
\ee
we have
\be
 \dist(0, \sigma(H_{h,x,\xi(x)}^{n + [-N_1,N_1]}))\text{ is not too small}.
\ee
This is what is called {\em elimination of double resonances}.
Using this, we are able to show that for many $x$,
we have
\be
 0\text{ is extremely close to the spectrum of } H_{h, x, \xi_1(x)}^{[-N_2, N_2]}.
\ee
Then, we construct $\xi_2$ similarly to $\xi_1$ and repeat
the process.

In order to show that $E_0$ is indeed an eigenvalue of $H_{h,x,y}$
for some $y$, we will show that there are many $x$ that are good
for all $j$. Furthermore, for these we have that
\be
 \xi_j(x) \to \xi_{\infty}(x)
\ee
and also the eigenfunctions $\psi_j$ corresponding to the eigenvalue $E_0$
are convergent to some $\psi_{\infty}$. This implies that
\be
 H_{h,x,\xi_{\infty}(x)}\psi_{\infty} = E_0 \psi_{\infty}
\ee
and thus that $E_0$ is an eigenvalue of $H_{h,x,y}$ for some $y$.

%
%
%

\section{Outline of the proof}

In this section, we explain the inductive construction,
which we use in the proof. The following two sections contain
the explanations of how to do check the initial condition
and how to obtain the induction step. The
proof is concerned with parameterizing isolated eigenvalues

\begin{definition}\label{def:epsisolated}
 Let $\eps > 0$, $A$ a self-adjoint operator, and $E \in \R$.
 $E$ is an $\eps$-isolated eigenvalue of $A$, if
 \be
  \sigma(A) \cap [E - \eps, E + \eps] = \{E \}
 \ee
 and $E$ is simple.
\end{definition}

An eigenvalue $E$ of $A$ is simple if $\ker(A - E)$ is one-dimensional.
In the setting of one-dimensional Schr\"odinger operators
this condition is always satisfied, see Remark~1.10 in \cite{tjac}.

Isolated eigenvalues are important, since they behave well
under perturbations. This can for example be seen in
Lemma~\ref{lem:perturbisolated}. The conclusions of this
lemma can be summarized that a $\eps$-isolated eigenvalue
is stable under perturbations of $A$ of size $\eps^2$.
Next, we define what we mean by a parametrization:

\begin{definition}
 Let $\xi: \T\to\T$, $\mathfrak{X}\subseteq\T$
 be a continuously differentiable function, $\eps > 0$,
 $L \in (0,\frac{1}{3})$, and $M \geq 0$.

 We say that $(\xi,\mathfrak{X})$ is a $(\eps,L)$-parametrization
 of the eigenvalue $E_0$ of $H^{[-M,M]}_{h,\bullet}$, if
 \begin{enumerate}
  \item For $x\in\mathfrak{X}$, we have
   \be
    E_0\text{ is a $\eps$-isolated eigenvalue of } H^{[-M,M]}_{h,x,\xi(x)}
   \ee
  \item We have $|\mathfrak{X}| \geq \frac{1}{\sqrt{\max(M,1)}}$.
  \item We have $\|\xi'\|_{L^{\infty}(\T)} \leq L$.
 \end{enumerate}
\end{definition}

In order to examine this definition, let us look
at the simplest example of $M = 0$. Then, we have
that $H^{[-0,0]}_{h,x,y}$ is just the multiplication
operator by $f(y)$. Hence, if we define
\be\label{eq:defxi0X0}
 E_0 = f(y_0),\quad \xi_0(x) = y_0,\quad \mathfrak{X}_0 = \T,
  \quad M_0 = 0,\quad \eps_0 = 1
\ee
we have that
\be
 (\xi_0, \mathfrak{X}_0)\text{ is a $(\eps_0,0)$-parametrization of 
  the eigenvalue $E_0$ of } H^{[-M_0,M_0]}_{h,\bullet}.
\ee
The essential part of the induction scheme will be to
show that given a parametrization at scale $M$, we can
extend it to a parametrization at scale $R \approx \E^{M^{c}}$
for some positive $c >0$. However, this alone will not
carry a sufficient amount of informations, we will also
want that the eigenfunctions will have something to do
with each other. For this, we introduce

\begin{definition}\label{def:extendpara}
 Let $(\xi_j, \mathfrak{X}_j)$ be $(\eps_j,L_j)$-parameterizations
 of the eigenvalue $E_0$ of $H^{[-M_j,M_j]}_{h, \bullet}$
 for $j = 1,2$.
  
 $(\xi_2, \mathfrak{X}_2)$ is said to be a $\delta$-extension 
 of $(\xi_1, \mathfrak{X}_1)$, if
 \begin{enumerate}
  \item $\mathfrak{X}_2 \subseteq \mathfrak{X}_1$.
  \item $\eps_2 < \eps_1$, $M_2 \geq M_1$.
  \item $L_2 \leq L_1 + \delta$ and $\|\xi_1 - \xi_2\|_{L^{\infty}(\mathfrak{X}_2)} \leq \delta$.
  \item Let $x \in \mathfrak{X}_2$ and $\varphi_j \in \ell^2([-M_j,M_j])$
   normalized eigenfunctions of $H^{[-M_j,M_j]}_{h,x,\xi_j(x)}$ corresponding
   to the eigenvalue $E_0$. We have for some $|a|=1$
   \be
    \|\varphi_1 - a \varphi_2\|_{W} \leq \delta.
   \ee
 \end{enumerate}
\end{definition}

Here, we take $W(n) = 1 + n^2$ and we define the
norm
\be
 \| u \|_{W} = \left(\sum_{n\in\Z} W(n) |u(n)|^2 \right)^{\frac{1}{2}}
\ee
which is always well defined, since for us $u$ and $v$
are non-zero for only finitely many $n$. The reason for
adding the weight $W$ is that, we will want to control
$\spr{\psi}{\partial_x V \psi}$, where the norm of $\partial_x V$
as an operator on $\ell^2([-N,N])$ grows like $N$.

The operators $\partial_x V$ and $\partial_y V$ are defined
by
\begin{align}
 (\partial_x V_{x,y} u)(n)&= n f'(y + n x + n(n-1) \alpha) u(n), \\
 (\partial_y V_{x,y} u)(n) &= f'(y + n x + n(n-1) \alpha) u(n),
\end{align}
where the $(x,y)$ are always the ones so that the $\psi$
in $\spr{\psi}{\partial_x V \psi}$ is an eigenfunction
of $H_{h,x,y}^{[-M,M]}$.

We define in the following
\be\label{eq:defdfprime}
 d = \frac{1}{10} \max(|f'(y_0)|, 1).
\ee
In order to understand one way, in which we will use
Definition~\ref{def:extendpara}, we prove

\begin{lemma}\label{lem:sprstable}
 Let $(\xi,\mathfrak{X})$ be a $\delta$-extension of $(\xi_0, \mathfrak{X}_0)$.
 Let $x \in \mathfrak{X}$ and $\psi$ be a normalized eigenfunction of 
 $H^{[-M,M]}_{h,x,\xi(x)}$ corresponding to the eigenvalue $E_0 = f(y_0)$.
 Assume that $\delta \leq \frac{d^5}{2 C_1}$  where $C_1 = 10 \|f'\|_{L^{\infty}(\T)}$.
 Then we have that
 \be
  |\spr{\psi}{V_x \psi}| \leq \frac{1}{4} d^{5},\quad
  |\spr{\psi}{V_y \psi}| \geq 4 d.
 \ee
\end{lemma}

\begin{proof}
 Let $x\in\mathfrak{X}$. By condition (iv) of Definition~\ref{def:extendpara},
 there exists a choice of a normalized eigenfunction $\psi$ of $H^{[-M,M]}_{h,x,\xi(x)}$
 corresponding to the eigenvalue $E_0$ such that
 \[
  \psi(0) \geq 1 - \delta,\quad \sum_{n\in\Z} |n| |\psi(n)|^2 \leq \delta.
 \]
 This implies that
 \[
  |\spr{\psi}{V_x \psi}| \leq C_1 \delta,\quad
  |\spr{\psi}{V_y \psi}| \geq 10 d - (C_1 + 10) \delta.
 \]
 The claim now follows by the choice of $\delta$.
\end{proof}

This lemma shows, how we will use the condition the parametrization
extends $(\xi_0, \mathfrak{X}_0)$. We will use these conditions
to obtain some control on the $\xi$ in the parametrization,
in particular that $\|\xi'\|_{L^{\infty}(\T)} \leq \frac{1}{3}$,
which we need to eliminate double resonances.

We will pass from scale to scale using the next theorem.

\begin{theorem}\label{thm:paratonextscale}
 Let $M$ be large enough and assume that the large deviation estimate
 hold. Furthermore, assume for $\eps = \E^{- M^{\frac{1}{50}}}$ that
 \be
  (\xi, \mathfrak{X})\text{ is a $(\eps, L)$-parametrization of the eigenvalue $E_0$ of }
   H_{h,\bullet}^{[-M,M]}
 \ee
 that $\frac{d^5}{2C_1}$-extends $(\xi_0,\mathfrak{X}_0)$
 and $L + \eps \leq \frac{1}{3}$. 
 Define $R = \lfloor \E^{M^{\frac{1}{1000}}} \rfloor$.
 Then there exists $(\hat{\xi}, \widehat{\mathfrak{X}})$ such that
 \be
  (\hat{\xi}, \widehat{\mathfrak{X}})\text{ is a $(\frac{1}{1000} \eps, \hat{L})$-parametrization
   of the eigenvalue $E_0$ of } H^{[-R,R]}_{h,\bullet}
 \ee
 with $\hat{L} = L + \eps$ and for $\eta = \E^{-\frac{1}{100} M}$
 \be
  (\hat{\xi}, \widehat{\mathfrak{X}})\text{ is a $\eta$-extension
   of }(\xi, \mathfrak{X})\text{ to scale }R.
 \ee
\end{theorem}

Since $M$ needs to be large, $(\xi_0, \mathfrak{X}_0)$ does
not satisfy the assumptions of this theorem. So, we will need

\begin{theorem}\label{thm:initcond}
 Let $M \geq 1$. Then there exists $h_2 = h_2(M, f, \delta) > 0$
 such that for $0 < h < h_2$ and $E_0 \in \mathcal{E}_{\delta}$, there exists 
 a $h^{\frac{1}{500}}$-parameterization $(\xi_1, \mathfrak{X}_1)$ at scale $M$ that $h^{\frac{1}{10}}$-extends
 $(\xi_0, \mathfrak{X}_0)$.
\end{theorem}

We now begin the proof of Theorem~\ref{thm:int1}. Choose $M$ so large
that Theorem~\ref{thm:paratonextscale} holds, finitely many additional
largeness conditions might be imposed below.
For $M_1 = M$, we define a sequence
\be
 M_{j+1} = \left\lfloor \E^{(M_{j})^{\frac{1}{1000}}} \right\rfloor
\ee
and
\be
 \eps_j = \E^{ - (M_j)^{\frac{1}{50}}},\quad L_{j+1} = L_{j} + \eps_{j}.
\ee
For $M$ large enough, $\sum_{\ell=j+1}^{\infty} \eps_{\ell} \leq \eps_j$
and $\eps_1 \leq \frac{1}{6}$

We will now inductively construct $(\xi_j, \mathfrak{X}_j)$ such that
\be
 (\xi_j, \mathfrak{X}_j)\text{ is a $(\eps_j,L_j)$-parametrization of the eigenvalue
   $E_0$ of } H_{h,\bullet}^{[-M_j,M_j]}
\ee
and
\be
 (\xi_{j+1}, \mathfrak{X}_{j+1})\text{ is a $\eps_{j+1}$-extension of } (\xi_j, \mathfrak{X}_j).
\ee
We construct $(\xi_1, \mathfrak{X}_1)$ using Theorem~\ref{thm:initcond}.
We can require here that $h$ is small enough such that
\be
  h^{\frac{1}{500}} \leq \eps_1,\quad h^{\frac{3}{2}} \leq \frac{d^{5}}{10 C_1}.
\ee
We now see that the assumptions of Theorem~\ref{thm:paratonextscale} hold
for $(\xi_1, \mathfrak{X}_1)$, so we can construct $(\xi_2, \mathfrak{X}_2)$.
Using the following lemma, one can now construct $(\xi_{j+1}, \mathfrak{X}_{j+1})$
from $(\xi_{j}, \mathfrak{X}_j)$ using Theorem~\ref{thm:paratonextscale}.

\begin{lemma}
 Let $\ell < j$, then 
 \be
  (\xi_{\ell},\mathfrak{X}_{\ell})\text{ is a $2 \eps_{\ell}$-extension
   of } (\xi_j, \mathfrak{X}_j).
 \ee
\end{lemma}

\begin{proof}
 This follows by being an extension is transitive.
\end{proof}

We will obtain this way a sequence of compact subsets
\be
 \mathfrak{X}_1 \supseteq \mathfrak{X}_2 \supseteq
  \mathfrak{X}_3 \supseteq \mathfrak{X}_4 \supseteq
   \dots
\ee
of $\T$.
Since the $\mathfrak{X}_j$ are compact, we have for
some $x_{\infty}$ that
\be
 x_{\infty} \in \bigcap_{j=1}^{\infty} \mathfrak{X}_j.
\ee
Define $y_j = \xi_j(x_{\infty})$. We clearly also have
$y_j \to y_\infty$. By condition (iv) of Definition~\ref{def:extendpara}, we can choose eigenfunctions
$\psi_j$ of $H_{h,x,y_j}^{[-M_j,M_j]}$ corresponding to
the eigenvalue $E_0$, which form a Cauchy sequence.
We have
\be
 H_{h,x_{\infty},y_{\infty}} \psi_{\infty} = E_0 \psi_{\infty},
\ee
where $\psi_{\infty} = \lim_{j\to\infty} \psi_j$.
This finishes the proof of Theorem~\ref{thm:int1}.

%
%
%

\section{Proof of the initial condition}
\label{sec:initcond}

In this section, we prove the initial condition, that is
Theorem~\ref{thm:initcond}. In order to 
make the statements look nice, we introduce

\begin{definition}\label{def:appepsisolatedev}
 Let $A$ be a self-adjoint operator, $E_0 \in \R$, $\eps > 0$, and $\eta \in (0, \eps)$.
 $E_0$ is an $\eta$-approximate $\eps$-isolated eigenvalue of $A$, if
 there exists $\lambda$ such that
 \be
  \sigma(A) \cap [E_0 - \eps, E_0 + \eps] = \{\lambda\},
 \ee 
 $|E_0 - \lambda| \leq \eta$, and $\lambda$ is simple.
\end{definition}

A convenient choice for us will be $\eta = \eps^{10}$
for most of this work. However, leaving $\eta$ as an independent
parameter has a big advantage. If we consider
\[
 \eta < \ti{\eta} < \ti{\eps} < \eps,
\]
then we have that $\eta$-approximate $\eps$-isolated
implies $\ti{\eta}$-approximate $\ti{\eps}$-isolated.
Similarly to Definition~\ref{def:extendpara}, we will define
what it means for an approximately isolated eigenvalue to
extend an eigenvalue.

\begin{definition}\label{def:extendappev}
 We say that an isolated eigenvalue $E_0$ of $H^{[-M,M]}$
 $\eta$-approximately extends to a $\eps$-isolated eigenvalue of $H^{[-R,R]}$
 if
 \begin{enumerate}
  \item There exists an eigenvalue $\lambda$ of $H^{[-R,R]}$
   satisfying $|\lambda - E_0| \leq \eta$.
  \item $\sigma(H^{[-R,R]}) \cap [E_0 - \eps, E_0 + \eps] = \{\lambda\}$.
  \item Let $\psi$ be the eigenfunction of $H^{[-M,M]}$
   and $\varphi$ be the one of $H^{[-R,R]}$. Then
   for some $|a|=1$
   \be
    \|\psi - a \varphi\|_{W} \leq \eta.
   \ee
 \end{enumerate}
\end{definition}

We note that this definition implies that
\be
 E_0\text{ is an $\eta$-approximate $\eps$-isolated eigenvalue of } H^{[-R,R]}.
\ee
However, as noted in the last section, condition (iii) is crucial
to control various quantities needed in our multi-scale scheme.
At the end of this section, we will prove

\begin{theorem}\label{thm:initcond2}
 Given $M_1 \geq 1$, there exists $h_3 = h_3(f,M_1,\delta) > 0$
 such that for $0 < h < h_3$, $E_0 \in \mathcal{E}_{\delta}$,
 we have the following:
 There exists $\mathfrak{X}_1 \subseteq \T$ satisfying
 \be
  |\mathfrak{X}_1| \geq \frac{1000}{\sqrt{M_1}}
 \ee
 such that for $x \in \mathfrak{X}_1$, we have
 with $\eps_1 = h^{\frac{1}{500}}$ and $\eta_1 = h^{\frac{1}{4}}$ that
 for $x \in \mathfrak{X}_1$
 \be
  E_0\text{ $\eta_1$-approximately extends to an $\eps_1$-isolated eigenvalue of }
   H_{h,x,\xi_0(x)}^{[-M_1, M_1]}.
 \ee
\end{theorem}

We recall from \eqref{eq:defxi0X0} that $\xi_0(x) = y_0$.
In the following, we will refer to the fact described in the
previous theorem as
$(\xi_0, \mathfrak{X}_0)$ extends to an $\eta_1$-approximate $\eps_1$-parametrization
  on scale $M_1$ on $\mathfrak{X}_1$.
It should be clear how to generalize this definition to
the more general situation, we are interested in.
Furthermore, one can check that Lemma~\ref{lem:sprstable}
remains valid in this setting. We now
come to the last result, we need to prove the initial condition.
I already formulate it in the way, we will need it for
the inductive step.

\begin{theorem}\label{thm:extendpara}
 Let $(\xi,\mathfrak{X})$ be a $(\eps, L)$-parametrization of $E_0$
 for $H^{[-M,M]}_{h, \bullet}$ such that 
 \be
  (\xi,\mathfrak{X}) \text{ is a $\frac{d^5}{20 \|f'\|_{L^{\infty}(\T)}}$-extension of
 $(\xi_0, \mathfrak{X}_0)$}.
 \ee
 Assume
 \be
  (\xi,\mathfrak{X})\text{ extends to an $\eps^5 \cdot \eta$-approximate $\eps$-parametrization
   on scale $R$ on } \widetilde{\mathfrak{X}}
 \ee
 and $\hat{L} = L + \eps \leq \frac{1}{3}$.
 Then there exist $(\hat{\xi}, \widehat{\mathfrak{X}})$ such that
 \be
  (\hat{\xi}, \widehat{\mathfrak{X}})\text{ is a $(\frac{\eps}{2},\hat{L})$-parametrization
   of $E_0$ for } H^{[-R,R]}_{h, \bullet},
 \ee
 $|\widehat{\mathfrak{X}}| \geq \frac{1}{3} |\widetilde{\mathfrak{X}}|$, and
 \be
  (\hat{\xi}, \widehat{\mathfrak{X}})\text{ is a $2 \eta$-extension of }
  (\xi,\mathfrak{X})\text{ from scale $M$ to scale $R$}.
 \ee
\end{theorem}

The proof of this theorem will be given at a later point. 
Before giving the proof of Theorem~\ref{thm:initcond2},
we will give the proof of Theorem~\ref{thm:initcond}.

\begin{proof}[Proof of Theorem~\ref{thm:initcond}]
 This follows from the previous two theorems.
\end{proof}

We now begin the proof of Theorem~\ref{thm:initcond2}.
The eigenfunctions of $H^{[-M,M]}_{h,x,y}$ are approximately given by
\be\label{eq:defpsi0}
 \psi_{h,x,y}^{0}(n) = \begin{cases} 1, & n = 0; \\
  \frac{h}{E_0 - V_{x,y}(-1)}, & n=-1; \\
  \frac{h}{E_0 - V_{x,y}(1)}, & n= 1; \\
  0, &\text{otherwise}. \end{cases}
\ee
Define $\psi_{h,x,y} = \frac{1}{\|\psi_{h,x,y}^{0}\|} \psi_{h,x,y}^{0}$,
the normalized version of the above vector. It should be noted that
in order for \eqref{eq:defpsi0} to make sense, we need that
$V_{x,y}(\pm 1) \neq 0$. We will be able to ensure this
with the following lemma

\begin{lemma}\label{lem:choicecalX}
 Let $M \geq 1$, $y_0 \in \T$, and $E_0 = f(y_0)$.
 There exists $h_4 = h_4(M,f) >0$ such that for $0 < h < h_4$, we have
 \begin{enumerate}
  \item There exists $\mathcal{X}$ of measure $|\mathcal{X}| \geq \frac{1}{2}$.
  \item For $x \in \mathcal{X}$, $n \in [-M,M]\setminus\{0\}$, we have
   \be
    |V_{x,y_0}(n) - E_0| \geq h^{\frac{1}{1000}}.
   \ee
 \end{enumerate}
\end{lemma}

In order to prove this lemma, we need to recall some things
about analytic functions. Since $f$ is analytic there exist $F > 0$
and $\alpha > 0$ such that for every $E \in \R$ and $\eps > 0$
\be
 |\{x\in\T :\quad |f(x) - E| < \eps\}| \leq F \cdot \eps^{\alpha}.
\ee
Since $V_{x,y}(n) = f(y + n x + n(n-1)\alpha)$, this implies
that
\be\label{eq:initevcartan}
 |\{x\in\T :\quad |V_{x,y}(n) - E| < \eps\}| \leq F \cdot \eps^{\alpha}
\ee
for all $E \in \R$, $\eps > 0$, $y\in\T$, $n\in\Z\setminus\{0\}$.
For $n = 0$, \eqref{eq:initevcartan} fails, since
$V_{x,y}(0) = f(y)$.

\begin{proof}[Proof of Lemma~\ref{lem:choicecalX}]
 By \eqref{eq:initevcartan}, we can find a set $\mathcal{X}$ such that for 
 $x \in \mathcal{X}$, we have
 \[
  |V_{x,y_0}(n) - E_0| \geq h^{\frac{1}{1000}}
 \]
 for $n \in [-N,N] \setminus \{0\}$ and
 \[
  |\mathcal{X}| \geq 1 - 2 N F \cdot (h)^{\frac{\alpha}{1000}},
 \]
 so $|\mathcal{X}| \geq \frac{1}{2}$ for $h \leq \left(\frac{1}{4 N F}\right)^{\frac{1000}{\alpha}}$.
\end{proof}

Lemma~\ref{lem:choicecalX} implies

\begin{lemma}
 Let $x \in \mathcal{X}$, then
 \be
  \|(H^{[-M,M]}_{h,x,y_0} - E_0) \psi_{h,x,y}^{0}\| = \sqrt{6} h^{1 - \frac{1}{1000}},
  \quad 1 \leq \|\psi_{h,x,y}^{0}\| \leq 1 + h^{2 - \frac{1}{500}}.
 \ee
\end{lemma}

\begin{proof}
 We have $|\psi_{h,x,y}^{0}(\pm 1)| \leq h^{1 - \frac{1}{1000}}$.
 The first inequality follows by the some computations,
 for the second one notice
 $\|\psi_{h,x,y}^{0}\| \leq \sqrt{1 + 2 h^{2 (1 - \frac{1}{1000})}}$,
 which implies the claim since $\sqrt{1 + t} \leq 1 + \frac{t}{2}$.
\end{proof}

Hence, we have that $H^{[-M,M]}_{h,x,y_0}$ has an eigenvalue $\lambda_x$,
that satisfies $|\lambda_x - E_0| \leq \sqrt{h}$, possibly imposing a new
smallness condition on $h$.
Denote by $E_j$ the eigenvalues of $H^{[-M,M]}_{h,x,y_0}$ and by
$\varphi_j$ the corresponding eigenfunctions. 
The previous lemma implies that there exists $\ell$ such that
\be
 |E_0 - E_{\ell}| \leq \sqrt{h}.
\ee

\begin{lemma}
 Assume $h$ is small enough. Then for $j \neq \ell$
 \be
  |E_0 - E_j| \geq h^{\frac{1}{500}}.
 \ee
\end{lemma}

\begin{proof}
 Consider the operator $\widehat{H}$ which is defined to be equal to
 $H^{[-M,M]}_{h,x,y_0}$, except that we replace $V(0)$ by $42 + E_0$.
 We have that
 $$
  \sigma(\widehat{H}) \cap [E_0 - h^{\frac{1}{1000}} + h, E_0 + h^{\frac{1}{1000}} - h] = \emptyset.
 $$
 Since $H^{[-M,M]}_{h,x,y_0} - \widehat{H}$ is a rank one
 operator, the claim follows.
\end{proof}

We summarize the findings so far as
\be
 E_0\text{ is a $\sqrt{h}$-approximate $h^{\frac{1}{500}}$-isolated eigenvalue of }
  H^{[-M,M]}_{h,x,y_0}
\ee
for $x\in\mathcal{X}$. It remains to check condition (iii)
of Definition~\ref{def:extendappev}.
Since the $\varphi_j$ form an orthonormal basis,
we can write $\psi_{h,x,y_0} = \sum_{j} \spr{\varphi_j}{\psi_{h,x,y_0}} \varphi_j$.
We have

\begin{lemma}
 For $x\in\mathcal{X}$, we have that
 \be
  \sum_{j \neq \ell} |\spr{\varphi_j}{\psi_{h,x,y_0}}|^2 \leq h^{\frac{1}{2}}.
 \ee
\end{lemma}

\begin{proof}
 A computation shows that for $x \in \mathcal{X}$
 $$
  \sum_j |E_j - E_0|^2 |\spr{\varphi_j}{\psi_{h,x,y_0}}|^2
   = \|(H_{h,x,y_0}^{[-M,M]} - E_0) \psi_{h,x,y_0}\|^2 \leq 6 h^{2 (1 - \frac{1}{1000})}.
 $$
 By the previous lemma, we have for $j \neq \ell$
 that $|E_j - E_0|^2 \geq \frac{1}{2} h^{\frac{1}{2}}$.
 The claim follows by some computations.
\end{proof}

We now come to

\begin{proof}[Proof of Theorem~\ref{thm:initcond2}]
 The previous lemma implies that
 \[
  \|\varphi_{\ell} - a \psi_{h,x,y}\| \leq h^{\frac{1}{4}}
 \]
 for some $|a| = 1$. This finishes the proof.
\end{proof}

%
%
%


%
%
%

\section{The proof of the multi-scale step; Theorem~\ref{thm:paratonextscale}}

I will begin by introducing the notion of {\em suitability}
for a Schr\"odinger operator $H$ acting on $\ell^2(\Z)$,
and then introduce the large deviation estimates. After
this, I will discuss the proof of Theorem~\ref{thm:paratonextscale}.
Given an interval $[a,b] \subseteq\Z$,
we introduce $H^{[a,b]}$ as the restriction
of $H$ to $\ell^2(\{a,\dots,b\})$. For $E\in\R$
and $k, \ell \in [a,b]$, we introduce the
{\em Green's function} by
\be
 G^{[a,b]}(E,k,\ell) =
  \spr{e_k}{ \left(H^{[a,b]} - E\right)^{-1} e_\ell},
\ee
where $\{e_\ell\}_{\ell\in\Z}$ denotes the standard
basis of $\ell^2(\Z)$. That is 
\be
 e_{\ell}(n) = \begin{cases} 1, & \ell=n; \\ 0,&\text{otherwise}.\end{cases}
\ee
We will quantify the properties
of the Green's function with the following definition.

\begin{definition}\label{def:suitable}
 Let $\gamma > 0$, $\Gamma > 1$, and $p \geq 0$. An
 interval $[-N,N]$ is called $(\gamma,\Gamma,p)$-suitable
 if the following hold
 \begin{enumerate}
  \item $\Gamma \leq \gamma N$.
  \item We have
   \be
    \left\|\left(H^{[-N,N]} - E\right)^{-1}\right\|
     \leq \frac{1}{2^{p}}\E^{\Gamma}.
   \ee
  \item For $k \in \{-N,N\}$ and $- \frac{2}{3} N \leq \ell \leq \frac{2}{3} N$,
   we have 
   \be
    \left|G^{[-N,N]}(E,k,\ell)\right| 
     \leq \frac{1}{2^{p}} \E^{-\gamma |k -\ell|}.
   \ee
 \end{enumerate}
\end{definition}

Lemma~\ref{lem:suitstable} shows that this definition has
a certain stability under perturbing $H$.
The next definition is again for the specific operator
$H_{h,x,y}$ defined in \eqref{eq:defHhxy}.

\begin{definition}
 Let $\gamma >0$, $\Gamma > 1$, $p \geq 0$, $N\geq 1$, and $E\in\R$.
 The set of unsuitability $\mathcal{U}_{h,E,\gamma,\Gamma,p}^{[-N,N]}$
 denotes the set of all $(x,y)\in \T^2$ such that
 \be
  [-N,N]\text{ is not $(\gamma,\Gamma,p)$-suitable for } H_{h,x,y} - E.
 \ee
\end{definition}

Lemma~\ref{lem:strucUy} will derive a certain geometric structure
for the set $\mathcal{U}$, whereas the next theorem shows that
the measure of this set is small.

\begin{theorem}\label{thm:ldtgreen}
 Assume that the Diophantine condition \eqref{eq:conddiop} holds.
 There exist $h_0 = h_0(f, c) > 0$, $\gamma \geq 1$ such that
 for $N \geq 100$, $E \in\R$, and $0 < h < h_0$
 \be
  |\mathcal{U}_{h,E,\gamma, \frac{1}{2} \gamma N,5}^{[-N,N]}|
   \leq \E^{-N^{\frac{1}{10}}}.
 \ee
\end{theorem}

\begin{proof}
 A proof of this theorem can be extracted from
 \cite{bbook}, \cite{b2002}, \cite{bgs}.
\end{proof}

We will refer to the assumption
that the conclusions of the previous theorem hold,
as that {\em the large deviation estimates} hold. 
The specific form of the constants in the theorem
above is of little importance to us, in particular
the exponent $\frac{1}{10}$ in $\E^{-N^{\frac{1}{10}}}$
could be any number $> 0$. I  am using this concrete value
to reduce the number of constants in the proofs.
The next theorem is the basic conclusion, we will draw from it.

\begin{theorem}\label{thm:nodblreson}
 Assume that the large deviation estimates hold,
 and that $M$ is large enough.
 Let $N = \lfloor M^{\frac{1}{100}} \rfloor$, $R =\lfloor \E^{M^{\frac{1}{1000}}} \rfloor$,
 and $\xi: \T\to\T$ satisfying $\|\xi'\|_{L^{\infty}(\T)}\leq\frac{1}{3}$.
 Then there exists $\mathcal{B}\subseteq\T$ satisfying
 \be
  |\mathcal{B}| \leq \frac{1}{M^{\frac{3}{4}}}
 \ee
 such that for $\frac{M}{10} \leq |n| \leq R$, we have
 \be
  n + [-N,N]\text{ is $(\gamma, \gamma N, 2)$-suitable for } H_{h,x,\xi(x)} - E_0
 \ee
 for $x \in \T\setminus\mathcal{B}$.
\end{theorem}

It should be pointed out here, that this theorem does not tell
us that double resonances happen with small probability.
It tells us that resonances happen with small probability
along curves satisfying certain estimates. We will
use these curve to parametrize resonances of $H^{[-M,M]}_{h,x,y}$.
This way, we can eliminate the double resonances relevant to
us. It would be interesting to obtain a true double resonance
elimination theorem, since it would imply Anderson localization.
I will still refer to the previous result as
{\em double resonance elimination}.

I also wish to point out that the geometric content of 
Theorem~\ref{thm:nodblreson} should be surprising. The information
that the large deviation estimates hold, tells us that 
the measure of a subset of $\T^2$ is small, then the output
tells us that certain curves intersect this set with
small probability. This is possible since, the set
$\mathcal{U}$ has further geometric structure and the
availability of a fast variable.

\bigskip

Having eliminated double resonances, we have the following
result, which tells us that eigenvalues extend

\begin{theorem}\label{thm:contev}
 Let $40 \E^{-\frac{1}{5} \gamma M} \leq \eps \leq \E^{-3\gamma N}$.
 Assume for $\frac{M}{10} \leq |n| \leq R$, we have
 \be
  n + [-N,N]\text{ is $(\gamma, \gamma N, 2)$-suitable for } H - E_0
 \ee
 and
 \be
  E_0\text{ is a $\eps$-isolated eigenvalue of } H^{[-M,M]}.
 \ee
 Then for $\eta = 2 \E^{-\frac{\gamma}{5} M}$
 \be
  E_0\text{ extends to an $\eta$-approximate $\frac{\eps}{1000}$-isolated eigenvalue of } H^{[-R,R]}.
 \ee
\end{theorem}

The proof of this theorem directly follows from the
more abstract Theorem~\ref{thm:contsingleev}.
We have no provided the abstract methods for proving 
Theorem~\ref{thm:paratonextscale}. We need

\begin{lemma}
 We have that
 \be
  \|\xi'\|_{L^{\infty}(\T)} \leq \frac{1}{3}.
 \ee
\end{lemma}

\begin{proof}
 We have that $\xi_0'(x) = 0$ for all $x$. Since
 $(\xi,\mathfrak{X})$ is a $\frac{1}{3}$-extension of
 $(\xi_0,\mathfrak{X}_0)$, the claim follows.
\end{proof}

\begin{proof}[Proof of Theorem~\ref{thm:paratonextscale}]
 Apply Theorem~\ref{thm:nodblreson} and introduce 
 \[
  \mathfrak{X}_1 = \mathfrak{X} \setminus \mathcal{B},
   \quad |\mathfrak{X}_1| \geq \frac{1}{2} |\mathfrak{X}|.
 \]
 By Theorem~\ref{thm:contev}, we have for $x\in\mathfrak{X}_1$
 and $\eta = \E^{-\frac{1}{5} M}$ that
 \[
  (\xi, \mathfrak{X})\text{ extends to an $\eta$-approximate $\frac{\eps}{1000}$-parametrization
    of $E_0$ for } H^{[-R,R]}_{h, \bullet}.
 \]
 By Theorem~\ref{thm:extendpara}, we obtain a parametrization
 for $(\hat{\xi},\widehat{\mathfrak{X}})$ such that
 \[
  |\widehat{\mathfrak{X}}| \geq \frac{1}{6} |\mathfrak{X}|.
 \]
 The claims follows using that $\| \varphi \|_{W} \leq (1 + R^2) \| \varphi \|$.
\end{proof}

%
%
%


%
%
%

\section{Suitability}
\label{sec:suitable}

In this section, we discuss the notion of suitability
defined in Definition~\ref{def:suitable} in more detail.
We begin with the following stability result

\begin{lemma}\label{lem:suitstable}
 Assume $\|\ti{H}^{[-N,N]} - H^{[-N,N]}\| \leq \frac{1}{2^{p+2}} \E^{-3\gamma N}$
 and 
 \be
   [-N,N]\text{ is $(\gamma,\Gamma,p+1)$-suitable for }H -E.
 \ee
 Then 
 \be
  [-N,N]\text{ is $(\gamma,\Gamma,p)$-suitable for }\ti{H} - E.
 \ee
\end{lemma}

\begin{proof}
 This is Lemma~5.3. in \cite{kcorr}.
\end{proof}

For the mechanism to eliminate double resonances, we will
need to understand the horizontal slices of the set
of unsuitability. For $y \in\T$, we introduce
\be
 U(y) = \{x:\quad (x,y) \in U\}
\ee
for any set $U \subseteq\T^2$. We will need

\begin{lemma}\label{lem:strucUy}
 For $N$ large enough, there exists a set $U$ such that
 \be
  \mathcal{U}_{h,E,\gamma,\frac{1}{2} \gamma N,3}^{[-N,N]}
   \subseteq U \subseteq
    \mathcal{U}_{h,E,\gamma,\gamma N,5}^{[-N,N]}
 \ee
 and for $y \in\T$
 \be
  U(y)\text{ consists of less then $N^{10}$ intervals}.
 \ee
\end{lemma}

Denote by $\hat{f}(k)$ the Fourier coefficients of $f$,
that is
\[
 f(x) = \sum_{k\in\Z} \hat{f}(k) e(k\cdot x)
\]
with $e(x) = \E^{2\pi\I x}$. 
Introduce $f^{R}(x) = \sum_{k=-R}^{R} \hat{f}(k) e(k\cdot x)$
and by $H_{h,x,y,R}$ the operator with potential
\be
 V_{x,y,R}(n) = f^{R}( (T^n(x,y))_2 ) 
  = f^{R}(y + n x + n (n-1)\alpha).
\ee
Since $f$ is analytic, we have for some positive $c > 0$ that
$\| f - f^{R} \|_{L^{\infty}(\T)} \leq \E^{-c R}$
for $R$ large enough. Define $U$ to be the set
of all $x$ such that
\begin{enumerate}
 \item We have
  \be
   \|(H_{h,x,y,R}^{[-N,N]} - E)^{-1}\|_{\mathrm{HS}} > \frac{1}{2^{p}} \E^{\frac{1}{2} \gamma N}
  \ee
 \item Condition (iii) of Definition~\ref{def:suitable} holds.
\end{enumerate}
Here $\|. \|_{\mathrm{HS}}$ denotes the Hilbert--Schmidt norm,
that is
\be
 \|A \|_{\mathrm{HS}} = \left(\sum_{k=1}^{L} \sum_{\ell=1}^{L} |A_{k,\ell}|^2\right)^{\frac{1}{2}},
\ee
where $A$ is a $L\times L$ matrix. We note that
\be
 \|A\| \leq \|A\|_{\mathrm{HS}} \leq L \| A\|.
\ee
Choosing $R = N^2$, the inclusions
of the sets follow by Lemma~\ref{lem:suitstable}.

\begin{proof}[Proof of Lemma~\ref{lem:strucUy}]
 It remains to discuss the bound on the number of intervals
 of the sections $U(y)$. By construction and choice of $R$,
 we have that for $|n| \leq N$
 \[
  \deg(V_{x,y,R}(n)) \leq N^3,
 \]
 where $\deg(.)$ denotes the degree of a trigonometric
 polynomial in $x$. Since, using Cramer's rule, one can
 rewrite the conditions defining $U$ as less than $10N$
 conditions involving polynomials of degree $\leq 5 N$
 in the $V_{x,y,R}(n)$, the claim follows.
\end{proof}

%
%
%


%
%
%

\section{Elimination of the fast variable}
\label{sec:elifast}

In this section, I discuss a variant of the frequency elimination
argument from the work \cite{bg} by Bourgain and Goldstein. There
are two differences. First for us $y$ will not be fixed, but depend
on $x$. Second, we have some additional terms, since we vary $x$,
the fast variable of the skew-shift, and not the frequency
of a rotation.

\begin{proposition}\label{prop:freqelim3}
 Let $U \subseteq \T^2$ and assume that
 for $y \in\T$, we have that
 \be
  U(y) = \{x:\quad (x,y) \in U\}
 \ee
 consists of at most $M$ intervals.
 Furthermore let $\xi: \T \to \T$ be a continuously
 differentiable function satisfying for $x \in \T$
 \be\label{eq:conddxi}
 |\xi'(x)| \leq \frac{1}{3}.
 \ee
 Then for $R \geq 1$, we have
 \be
  |\{x \in\T:\quad \exists\ \ell \sim R:\quad T^{\ell}(x,\xi(x)) \in U\}|
    \leq 120 R^4 \sqrt{|U|} + \frac{2 M}{R}.
 \ee
\end{proposition}

Here we denote by $\ell \sim R$ that $R \leq \ell \leq 2 R$.
If one were to relax \eqref{eq:conddxi} to $|\xi'(x)| \leq L$,
one would need to replace the meaning of this to
$\ell = R + 3 r \lceil L \rceil$ for $0 \leq r \leq R -1$.
For simplicity, I have decided to work with \eqref{eq:conddxi}.

Before proving Proposition~\ref{prop:freqelim3},
we show how it implies Theorem~\ref{thm:nodblreson}.

\begin{proof}[Proof of Theorem~\ref{thm:nodblreson}]
 Apply the previous proposition to the set $U$
 constructed in Lemma~\ref{lem:strucUy}. We have that
 \[
  |U| \leq \E^{-M^{\frac{1}{1000}}}
 \]
 and that the sections $U(y)$ consist of less then $M^{\frac{1}{10}}$
 many intervals. Define a sequence $R_j$ by
 \[
  R_j = 2^{j - 1} \frac{M}{10}
 \]
 for $j=1,\dots, j_{max}$, where $j_{max}$ is defined
 to be minimal such that $R_{j_{max}} > R$. We can apply
 Proposition~\ref{prop:freqelim3} to $U$ and these $R_j$
 and see that the claim holds as long as $M$ is large
 enough.
\end{proof}

The main idea of the proof of Proposition~\ref{prop:freqelim3}
is that the second coordinate
of $T^{\ell}(x, \xi(x))$ is $\ell$ to $1$, whereas the
first one is $1$ to $1$. This create sufficient independence
of the two coordinates to imply the claim.

We begin now with fleshing out the details.
Define for $x \in \T$ and $\ell \sim R$
\be
 (\varphi_{\ell}(x), \psi_{\ell}(x)) = T^{\ell}(x,\xi(x)),
\ee
so that
\[
 \varphi_{\ell}(x) = x + 2 \ell \alpha, \quad
 \psi_{\ell} (x) =\xi (x) + \ell x + \ell (\ell-1) \alpha.
\]
We will need

\begin{lemma}
 The map $\psi_{\ell}$ is $\ell$ to $1$
 and satisfies for $x \in \T$
 \be
  \ell - \frac{1}{3} \leq \psi_{\ell}'(x) \leq \ell + \frac{1}{3}.
 \ee
\end{lemma}

\begin{proof}
 Since $|\xi'(x)| \leq \frac{1}{3}$, we have that for any $x,y \in\T$
 \[
  |\xi(x) - \xi(y)| \leq \frac{1}{3}.
 \]
 The claim about $\psi_{\ell}$ being $\ell$ to $1$ follows.
 The claim about the derivative is a computation.
\end{proof}

By the previous lemma, there exist maps $\theta_{\ell,p}:\T\to\T$
such that for every $x \in \T$, there exists a unique $1 \leq p \leq \ell$
such that
\be
 x = \theta_{\ell,p}(\psi_{\ell}(x)).
\ee
From this, we have for these that
\be\label{eq:estithetaprime}
 \frac{1}{\ell + \frac{1}{3}} \leq 
 \theta_{\ell, p}'(y) \leq \frac{1}{\ell - \frac{1}{3}}
\ee
and that for any $y\in\T$ and $1 \leq p \leq \ell$
\be
 \psi_{\ell}(\theta_{\ell,p}(y))  = y.
\ee

\begin{lemma}
 We have
 \begin{align}
 \nn |\{x \in\T:\quad &\exists\ \ell \sim R:\quad T^{\ell}(x,\xi(x)) \in U\}|\\
    & \leq \frac{2}{R} \int_{\T} 
   \#\left\{\begin{aligned} R &\leq \ell  \leq 2 R \\ 
    1 &\leq p \leq \ell \end{aligned}:\quad \varphi_{\ell}(\theta_{p, \ell}(y)) \in U(y) \right\}dy.
 \end{align}
\end{lemma}

\begin{proof}
 Since 
 \[
  \sum_{\ell=R}^{2 R} \chi_U(\varphi_{\ell}(x), \psi_{\ell}(x))= 
   \#\{\ell\sim R:\quad T^{\ell}(x,\xi(x)) \in U\}
 \]
 we have
 \[
  |\{x \in\T:\quad \exists\ \ell \sim R:\quad T^{\ell}(x,\xi(x)) \in U\}|
    \leq \sum_{\ell =R}^{2 R} \int_{\T} \chi_U(\varphi_{\ell}(x), \psi_{\ell}(x)) dx.
 \]
 Performing the change of variables $y = \psi_{\ell}(x)$, we obtain that
 \[
  \leq \frac{2}{R} \sum_{\ell=R}^{2 R} \sum_{p=1}^{\ell}
   \int_{\T} \chi_{U}(\varphi_{\ell}(\theta_{\ell,p}(y)), y) dy.
 \]
 The claim follows.
\end{proof}

A simple computation shows the estimate
\be
 \#\left\{\begin{aligned} R &\leq \ell  \leq 2 R \\ 
    1 &\leq p \leq \ell \end{aligned}:\quad \varphi_{\ell}(\theta_{p, \ell}(y)) \in U(y) \right\}
  \leq 2 R^2.
\ee
Given $\gamma >0$, define the set $\mathcal{B}_1^{\gamma}$ by
\be
 \mathcal{B}_1^{\gamma} = \{y:\quad |U(y)| > \gamma\}.
\ee
By Markov's inequality, we have that $|\mathcal{B}_{1}^{\gamma}| \leq \frac{1}{\gamma} |U|$
and
\be
 \int_{\mathcal{B}_{1}^{\gamma}} 
 \#\left\{\begin{aligned} R &\leq \ell  \leq 2 R \\ 
    1 &\leq p \leq \ell \end{aligned}:\quad \varphi_{\ell}(\theta_{p, \ell}(y)) \in U(y) \right\}
   dy
  \leq \frac{2}{\gamma} |U| \cdot R^2.
\ee
Define $\eta_{\ell,p}(y) = \varphi_{\ell}(\theta_{\ell,p}(y))$.
Define the set $\mathcal{B}_2^{\gamma}$ as the set of $y$
such that there exist $(\ell_1, p_1) \neq (\ell_2, p_2)$ such
that
\be\label{eq:diffetap1l1etap2l2}
 |\eta_{\ell_1,p_1}(y) - \eta_{\ell_2,p_2}(y)| \leq \gamma.
\ee

\begin{lemma}
 For $\gamma < \frac{1}{100 R}$, we have that
 \be
  |\mathcal{B}_{2}^{\gamma}| \leq 200 R^6 \gamma.
 \ee
\end{lemma}

\begin{proof}
 Consider for $(\ell_1, p_1) \neq (\ell_2, p_2)$ the set
 $\mathcal{Y}_{(\ell_1,p_1),(\ell_2,p_2)}$ of $y$ satisfying
 \eqref{eq:diffetap1l1etap2l2}.  If $\ell_1 = \ell_2$, we have
 that $\mathcal{Y}_{(\ell_1,p_1),(\ell_2,p_2)} = \emptyset$,
 because of the constructions
 of the functions $\theta_{\ell, p}$.

 So consider now $\ell_1 < \ell_2$ and define
 \[
  g(y) = \eta_{\ell_1,p_1}(y) - \eta_{\ell_2,p_2}(y).
 \]
 A computation using \eqref{eq:estithetaprime} shows that
 $|g'(y)|\geq \frac{1}{20 R^2}$. This implies that
 \[
  |\mathcal{Y}_{(\ell_1,p_1),(\ell_2,p_2)}| \leq 40 R^2 \gamma
 \]
 since the functions $\theta_{p,\ell}$ are increasing.

 Now, since there are less than $4R^4$ possible choices
 for $(\ell_1, p_1) \neq (\ell_2, p_2)$ the claim follows.
\end{proof}

The previous considerations imply that if we 
choose
\be
 \gamma = \frac{1}{10 R^3} \sqrt{|U|},
\ee
we obtain that
\be
 \int_{\mathcal{B}_{1}^{\gamma} \cup \mathcal{B}_{2}^{\gamma}} 
 \#\left\{\begin{aligned} R &\leq \ell  \leq 2 R \\ 
    1 &\leq p \leq \ell \end{aligned}:\quad \varphi_{\ell}(\theta_{p, \ell}(y)) \in U(y) \right\}
  dy
  \leq 60 R^5 \sqrt{|U|}.
\ee
The last piece is

\begin{lemma}
 Let $y \in \T\setminus(\mathcal{B}_{1}^{\gamma} \cup \mathcal{B}_{2}^{\gamma})$,
 then 
 \be
   \#\left\{\begin{aligned} R &\leq \ell  \leq 2 R \\ 
    1 &\leq p \leq \ell \end{aligned}:\quad \varphi_{\ell}(\theta_{p, \ell}(y)) \in U(y) \right\}
 \leq M.
 \ee
\end{lemma}

\begin{proof}
 By assumption and $y \notin \mathcal{B}_{1}^{\gamma}$, we may write
 \[
  U(y) = \bigcup_{j=1}^{\widehat{M}} I_j
 \]
 for $\widehat{M} \leq M$ and $I_j$ intervals of length $\leq \gamma$.
 Since $y \notin \mathcal{B}_{2}^{\gamma}$, we have that for each $j$
 there is at most one pair $(\ell, p)$ such that
 \[
  \varphi_{\ell}(\theta_{p,\ell}(y)) \in I_j.
 \]
 The claim follows.
\end{proof}

\begin{proof}[Proof of Proposition~\ref{prop:freqelim3}]
 This follows by the previous lemma and the
 equation preceding it.
\end{proof}

%
%
%

\section{Controlling a single eigenvalue}

This section develops a mechanism to keep control of a single eigenvalue,
when changing from scale to scale. The basic idea is the following:
A single of the restrictions of the operator has an eigenvalue in a given
energy interval, here $H^{[-M,M]}$ in the energy interval $[E_0-\eps,E_0+\eps]$.
If all other restrictions are suitable for this energy interval,
then the same holds for the larger interval.

\begin{theorem}\label{thm:contsingleev}
 Let $R \geq M \geq N \geq 1$, $E_0\in\R$, and $\eps > 0$.
 Assume
 \begin{enumerate}
  \item For $n \in [-R,R] \setminus [-\frac{M}{10}, \frac{M}{10}]$ with
   $[n-N,n+N]\subseteq [-R,R]$, we have
   \be
    [n-N,n+N]\text{ is $(\gamma,\gamma N,1)$-suitable for }H -E_0.
   \ee
  \item $40 \E^{-\frac{1}{5} \gamma M} \leq \eps \leq \E^{-3\gamma N}$.
  \item $H^{[-M,M]}$ has exactly one eigenvalue $\lambda_0 \in [E_0 - \eps, E_0 + \eps]$.
  \item This eigenvalue satisfies
   \be
    |\lam_0 - E_0| \leq \frac{\eps}{100}.
   \ee
 \end{enumerate}
 Then $H^{[-R,R]}$ has exactly one eigenvalue $E$ in $[E_0 - \frac{\eps}{2}, E_0 + \frac{\eps}{2}]$,
 which satisfies
 \be
  |E - \lam_0| \leq 2 \E^{-\frac{\gamma}{5} M},\quad
  |E - E_0| \leq \frac{\eps}{10}.
 \ee
 Furthermore, let $\varphi$ be a normalized eigenfunction of $H^{[-M,M]}$
 corresponding the eigenvalue $\lambda_0$. Then there exists a normalized
 eigenfunction $\psi$ of $H^{[-R,R]}$ corresponding to the eigenvalue $E$
 such that
 \be
  \|\varphi - \psi\| \leq 2\E^{-\frac{1}{10} \gamma M}.
 \ee
\end{theorem}

Before starting with the proof of this theorem, let us note

\begin{corollary}\label{cor:contsingleev}
 Let $\varphi$ and $\psi$ be as in the previous theorem. Then
 for any bounded operator $A$, we have
 \be
  |\spr{\psi}{A\psi} - \spr{\varphi}{A\varphi}| 
   \leq 4 \|A\| \E^{-\frac{1}{10} \gamma M}.
 \ee
\end{corollary}

\begin{proof}
 Write $\eta = \psi - \varphi$. We have $\|\eta\| \leq 2 \E^{-\frac{1}{10} \gamma M}$.
 A computation shows that
 \[
  \spr{\psi}{A\psi} - \spr{\varphi}{A\varphi} = \spr{\eta}{A\varphi} + \spr{\varphi}{A \psi}.
 \]
 Since $\varphi$ and $\psi$ are normalized,
 the claim follows.
\end{proof}

It should be pointed out at this point that the quantity $\spr{\psi}{A \psi}$
is independent of the normalized eigenfunction chosen. This follows
from that for two normalized eigenfunctions $\psi, \ti{\psi}$ of $H^{[-R,R]}$ to
the eigenvalue $E$, one has $\psi = c \ti{\psi}$ for some $|c| = 1$,
since the eigenspaces of $H^{[-R,R]}$ to a single eigenvalue are
one-dimensional.

Let us now begin with actually proving Theorem~\ref{thm:contsingleev}.
The following lemma is a consequence of Lemma~\ref{lem:suitstable}.

\begin{lemma}\label{lem:issuitevpf}
 Let $E \in [E_0 - \eps, E_0 + \eps]$ and 
 $n \in [-R,R] \setminus [-\frac{M}{10}, \frac{M}{10}]$ with
 $[n-N,n+N]\subseteq [-R,R]$. Then
 \be
  [n-N,n+N]\text{ is $(\gamma,\gamma N,0)$-suitable for }H - E.
 \ee
\end{lemma}

\begin{proof}
 This follows by Lemma~\ref{lem:suitstable}, since $\eps \leq \E^{-3\gamma N}$.
\end{proof}

Having this lemma at hand, we will now first show that
$H^{[-R,R]}$ has at least one eigenvalue in the interval
$[E_0-\frac{\eps}{2},E_0+\frac{\eps}{2}]$. This follows by
the next lemma, since $2 \E^{-\frac{1}{5} \gamma M} \leq \frac{\eps}{20}$.

\begin{lemma}
 Let $h \in (0,\frac{1}{2})$. There exists an eigenvalue $E$ of $H^{[-R,R]}$ in
 \be
  [\lam_0 - 2 \E^{-\frac{1}{5} \gamma M}, \lam_0 + 2 \E^{-\frac{1}{5} \gamma M}].
 \ee
\end{lemma}

\begin{proof}
 We recall that $H^{[-M,M]} \varphi = \lam_0 \varphi$.
 For $n \in [-(N-M),M-N] \setminus[-\frac{1}{10} M,\frac{1}{10} M]$,
 we have by Lemma~\ref{lem:issuitevpf}
 \[
  |\varphi(n)| \leq \E^{-\gamma N} \max(|\varphi(n+N+1)|, |\varphi(n-N-1)|),
 \]
 where $\varphi(n) = 0$ for $n \notin [-M,M]$. By iterating this equation,
 we find since $\frac{9}{10} M \cdot \frac{N}{N+1} \leq \frac{1}{4}$ that
 \[
  |\varphi(M)|, |\varphi(-M)| \leq \frac{1}{4} \E^{-\frac{1}{5} \gamma M}.
 \]
 By using $\varphi$ as a test function for $H^{[-R,R]}$, we have
 \[
  \| (H^{[-R,R]} - \lambda) \varphi \| \leq \E^{-\frac{1}{5} \gamma M}.
 \]
 This implies the claim.
\end{proof}

We now turn to show that $H^{[-R,R]}$ has only one eigenvalue in the
energy interval $[E_0 - \frac{\eps}{2}, E_0 + \frac{\eps}{2}]$.
Let now $E \in [E_0 - \frac{\eps}{2}, E_0 + \frac{\eps}{2}]$ be an eigenvalue of $H^{[-R,R]}$
and $\psi$ a corresponding normalized eigenfunction. Define the function $u$
by
\be
 u(n) = \begin{cases} \psi(n), &-M \leq n \leq M;\\
  0,&\text{otherwise}.\end{cases}
\ee
We have

\begin{lemma}
 Let $h \in (0,\frac{1}{2})$. Then
 \be\label{eq:normHMEu}
  \| (H^{[-M,M]} - E) u \| \leq \E^{-\frac{\gamma}{5} M}
 \ee
 and $\|u\|_{\ell^2([-M,M])} \geq 1 - 10 \E^{-\frac{\gamma}{5} M}$.
\end{lemma}

\begin{proof}
 By Lemma~\ref{lem:issuitevpf}, we have for $n \in [-R-M,R-M] \setminus [-\frac{M}{10}, \frac{M}{10}]$
 that
 $$
  |\psi(n)| \leq \E^{-\gamma N} \max(|\psi(n+N+1)|, |\psi(n-N-1)|)
 $$
 where $\psi(n) = 0$ for $|n| \geq R+1$. By iterating this, we can
 conclude that $\psi$ decays exponentially away from $[-\frac{M}{10}, \frac{M}{10}]$.
 In fact, one obtains that
 $$
  |\psi(n)| \leq \E^{- \ti{\gamma} (|n| - \frac{1}{10} M)},
 $$
 where $\ti{\gamma} = \gamma(1 - \frac{1}{N}) \geq \frac{\gamma}{2}$.
 \eqref{eq:normHMEu} follows as in the previous lemma. The claim
 on the norm of $u$ is similar.
\end{proof}

We may choose eigenvalue $\lambda_j$ and eigenfunctions 
$\varphi_j$ of $H^{[-M,M]}$ for $j \in [-M,M] \setminus \{0\}$,
which complete $\lambda_0$ and $\varphi_0 = \varphi$
to an orthonormal basis of $\ell^2(\{-M, \dots, M\})$.
Define $u_1 = \frac{1}{\|u\|} u$, which is now a normalized
function in $\ell^2(\{-M, \dots, M\})$, which satisfies
\be
 \|(H^{[-M,M]} - E) u_1\| \leq 20 \E^{-\frac{\gamma}{5} M}
\ee
by the previous lemma. We have that
\be
 \psi = \sum_{j=-M}^{M} \spr{\varphi_j}{u_1} \varphi_j.
\ee
We have that

\begin{lemma}
 For $j \neq 0$, we have
 \be
  |\spr{\varphi_j}{u_1}| \leq 20 \E^{-\frac{\gamma}{10} M}.
 \ee
\end{lemma}

\begin{proof}
 We have that
 $$
  \sum_{j} (E_j - E)^2 |\spr{\varphi_j}{\psi}|^2 = \| (H^{[-M,M]} - E) u_1\| \leq 400 \E^{-2 \frac{\gamma}{5} M}.
 $$
 Now the claim follows, since for $j \neq u_1$, we have
 $|E_j - E| \geq \frac{1}{2} \eps \geq \E^{-\frac{\gamma}{10} N}$.
\end{proof}

Hence, we have that
\be
 |\spr{\varphi_0}{u_1}|^2 = 1 - \sum_{j\neq 0} |\spr{\varphi_j}{u_1}|^2
   \geq 1 - 40 M \E^{-\frac{1}{10} \gamma M}.
\ee

\begin{proof}[Proof of Theorem~\ref{thm:contsingleev}]
 Since 
 \[
  \spr{\varphi_0}{\psi} = \|u\| \spr{\varphi_0}{u_1},
 \]
 we may achieve by replacing $\psi$ by $c \psi$, where $c = \frac{\ol{\spr{\varphi_0}{u_1}}}{|\spr{\varphi_0}{u_1}|}$
 that
 \[
  |\spr{\varphi_0}{\psi} - 1| \leq \E^{-\frac{1}{20} \gamma M}.
 \]
 Since
 \[
  \|\psi - \varphi_0\|^2 = \|\psi\|^2 + \|\varphi_0\|^2 - 2 \spr{\varphi_0}{\psi}
 \]
 the claim of the theorem follows.
\end{proof}

%
%
%


%
%
%

\section{Eigenvector perturbations}
\label{sec:evperturb}

In this section, I discuss how to parametrize eigenvalues in a small neighborhood.
I have decided to work in a somewhat general setting, in hope that this clarifies
some of the arguments. Consider for $t > 0$ the square
\be
 D_t = \{(x,y) \in \R^2:\quad |x| \leq t,\ |y| \leq t\} = [-t,t]^2.
\ee
We will be considering a continuously differentiable
family of self-adjoint operator $H(x,y)$ defined
for $(x,y) \in D_t$. We will always assume that
\be\label{eq:HxHyC}
 \|\partial_x H(x,y)\| \leq C,\quad \|\partial_y H (x,y) \| \leq C
\ee
for $(x,y) \in D_t$. We will be interested in small perturbations,
that is in the language above that $t$ is small. We begin with

\begin{theorem}\label{thm:evperturb1}
 Let $\eps > \eta > 0$ with $4\eta < \eps$.
 Assume \eqref{eq:HxHyC} and that
 \be\label{eq:E0etaepsevH00}
  E_0\text{ is an $\eta$-approximate $\eps$-isolated eigenvalue of } H(0,0).
 \ee
 Define $s = \frac{1}{32} \frac{\eta \eps}{C}$. Then for any $(x,y) \in D_s$, we have
 \be
  E_0\text{ is a $2 \eta$-approximate $\frac{1}{2} \eps$-isolated eigenvalue of } H(x,y).
 \ee
 Furthermore denote by $\psi_{x,y}$ the associated normalized eigenfunction of $H(x,y)$.
 Then for any $(x,y), (\ti{x}, \ti{y}) \in D_s$, there exists $|a| = 1$
 such that
 \be
  \|\psi_{x,y} - a \psi_{\ti{x}, \ti{y}}\| \leq \eta.
 \ee
\end{theorem}

We recall that $E_0$ is an $\eta$-approximate $\eps$-isolated eigenvalue
of $H(0,0)$, if there exists an eigenvalue $\lambda$ of
$H(0,0)$ such that
\be
 \sigma(H(0,0)) \cap [E_0 - \eps, E_0 + \eps] = \{\lambda\}
\ee
and $|\lambda - E_0| \leq \eta$.
See Definition~\ref{def:appepsisolatedev}.

Instead of proving Theorem~\ref{thm:evperturb1} directly,
I will first prove the following more abstract lemma.

\begin{lemma}\label{lem:perturbisolated}
 Let $A, B$ be self-adjoint operators,
 $E_0 \in\R$, $\eps > 0$. Assume that $E$ is
 a simple eigenvalue of $A$ such that
 \be
  \sigma(A) \cap [E_0 - \eps, E_0 + \eps] = \{E\},
   \quad |E - E_0| \leq \frac{1}{4} \eps
 \ee
 and that $\|A - B\| \leq t \eps$ with $t\in (0,\frac{1}{4})$.
 Then there exists a simple eigenvalue $\lambda$
 of $B$ such that $|\lambda - E| \leq t \eps$ and
 \be
  \sigma(B) \cap [E_0 - \frac{3}{4} \eps, E_0 + \frac{3}{4} \eps]
   = \{\lambda\}.
 \ee
 Furthermore, denote by $\psi$, $\varphi$ normalized vectors
 such that $A \psi = E \psi$, $B \varphi = \lambda \varphi$.
 Then, there exists $|a| = 1$ such that
 \be
  \|\varphi - a \psi\| \leq 8 t.
 \ee
\end{lemma}

\begin{proof}
 It follows since $\|A - B\| \leq t \eps$ and $t\in (0,\frac{1}{4})$ that
 \[
  \{\lambda\} = \sigma(H) \cap [E - \frac{3}{4} \eps, E + \frac{3}{4} \eps],
   \quad |\lambda - E| \leq \eps t.
 \]
 Define $\lambda_1= \lambda$ and denote by $\lambda_2, \dots, \lambda_M$
 an enumeration of the other eigenvalues. Denote by $\varphi_1, \dots, \varphi_M$
 a choice of corresponding normalized eigenfunctions.
 We have
 \[
  \|(B - \lam) \psi\| \leq \|(B-A) \psi\|+\|(\lam-E)\psi\|+\|(A-E)\psi\|
   \leq \eps t + \eps t + 0 \leq 2 \eps t
 \]
 and thus
 \[
  \|(B - \lam) \psi\|^2 = \sum_{j=2}^{M} (\lam - \lam_j)^2 |\spr{\psi}{\varphi_j}|^2 \leq 4 \eps^2 t^2.
 \]
 Furthermore, we have for $j \geq 2$
 \[
  |\lam_j - E_0| \geq \frac{3}{4} \eps,\quad |\lam - E_0| \leq \frac{1}{2} \eps
 \]
 and thus $|\lam - \lam_j| \geq \frac{1}{4} \eps$. It follows that
 $\sum_{j=2}^{M} |\spr{\psi}{\varphi_j}|^2 \leq 32 t^2$.
 This implies that $|\spr{\psi}{\varphi_1}| \geq \sqrt{1 -  16 t^2} \geq 1 - 32 t^2$.
 Choose
 \[
  \varphi = \frac{\ol{\spr{\psi}{\varphi_1}}}{|\spr{\psi}{\varphi_1}|} \varphi_1
 \]
 Then, we have $\| \psi - \varphi \|^2 = 2( 1 - \spr{\varphi}{\psi} ) \leq 64 t^2$.
\end{proof}

We now come to

\begin{proof}[Proof of Theorem~\ref{thm:evperturb1}]
 A computation shows that for $(x,y) \in D_s$ that
 \[
  \|H(x,y) - H(0,0) \| \leq C |x| + C |y| \leq \frac{1}{16} \eps \eta.
 \]
 The first claim now follows from the previous lemma.
 The furthermore statement follows from the furthermore
 statement of the previous lemma and that 
 $\|H(x,y) - H(\ti{x}, \ti{y})\| \leq \frac{\eta}{8} \eps$.
\end{proof}

Next, we come to

\begin{theorem}\label{thm:evperturb2}
 Let $\eps > \eta > 0$, $\delta \in (0,\frac{1}{3})$.
 Assume \eqref{eq:HxHyC} and \eqref{eq:E0etaepsevH00}.
 Furthermore let $\psi$ be a normalized eigenfunction of $H(0,0)$
 corresponding to the eigenvalue $E \in [E_0 - \eps, E_0 + \eps]$
 and assume
 \be\label{eq:sprdH00}
  \spr{\psi}{\partial_y H(0,0) \psi} \geq 2 \delta,\quad 
  |\spr{\psi}{\partial_x H (0,0) \psi}| \leq \frac{1}{2} \delta^2.
 \ee
 Define
 \be\label{eq:desperev2}
  s = \frac{1}{50} \cdot \frac{\eps \delta^2}{C^2}.
 \ee
 and also assume 
 \be\label{eq:EE0small}
  |E - E_0| \leq \frac{\delta s}{3}.
 \ee
 Then for $(x,y) \in D_{s}$,
 there exists a curve $\xi: [-s,s] \to \R$ such that
 \be
  E_0\text{ is a $\frac{\eps}{2}$-isolated eigenvalue of } H(x,\xi(x)),
 \ee
 $|\xi'(x)| \leq \delta$, and 
 \be
  |\xi(x)| \leq \frac{|E - E_0|}{\delta} + \delta s.
 \ee
\end{theorem}

We have

\begin{lemma}
 Let $\psi_{x,y}$ be the eigenfunction of $H(x,y)$ associated
 to the eigenvalue in $[E_0-\frac{\eps}{2}, E_0 +\frac{\eps}{2}]$.
 Then
 \be
  \spr{\psi_{x,y}}{\partial_y H(x,y) \psi_{x,y}} \geq \delta,\quad 
  |\spr{\psi_{x,y}}{\partial_x H (x,y) \psi_{x,y}}| \leq \delta^2.
 \ee
\end{lemma}

\begin{proof}
 Since $|x|, |y| \leq \frac{1}{24} \frac{\eps}{C} \cdot \frac{\delta^2}{2 C}$,
 we have by Theorem~\ref{thm:evperturb1} for some $|a| = 1$ that
 \[
  \|\psi_{x,y} - a \psi_{0,0}\| \leq \frac{\delta^2}{2 C}.
 \]
 Furthermore a computation shows for any normalized
 $\varphi, \psi$ and operator $A$
 \[
  |\spr{\psi}{A\psi} -\spr{\varphi}{A \varphi}| \leq 2 \|A\| \cdot \|\psi - \varphi\|.
 \]
 Now, the claim follows by \eqref{eq:sprdH00}.
\end{proof}

It is well known that the unique eigenvalue $\lambda(x,y)$
of $H(x,y)$ in $[E_0 - \frac{\eps}{2}, E_0 + \frac{\eps}{2}]$
is a continuously differentiable function of $x$ and $y$.
Furthermore, its derivatives are given by
\be
 \partial_x \lambda(x,y) = \spr{\psi_{x,y}}{\partial_x H(x,y) \psi_{x,y}},
 \quad
 \partial_y \lambda(x,y) = \spr{\psi_{x,y}}{\partial_y H(x,y) \psi_{x,y}}.
\ee

\begin{lemma}
 For $|x| \leq s$, there exists $\xi(x)$ such that
 \be
  \lambda(x,\xi(x)) = E_0.
 \ee
\end{lemma}

\begin{proof}
 By \eqref{eq:EE0small} and the previous lemma, we have that
 \[
  |\lambda(x,0) - E_0 | \leq \frac{2}{3} \delta s.
 \]
 Again by the previous lemma, we thus obtain that
 \[
  \lambda(x,s) \geq E_0 + \frac{1}{3} \delta s,\quad
   \lambda(x,-s) \geq E_0 - \frac{1}{3} \delta s.
 \]
 The claim follows since $\lambda(x,y)$ is continuous.
\end{proof}

We are now ready for

\begin{proof}[Proof of Theorem~\ref{thm:evperturb2}]
 The intermediate value theorem implies that there exists some
 $\xi(x)$ such that $\lambda(x,\xi(x)) = E_0$. It follows
 from the implicit function theorem that
 \[
  \xi'(x) = - \frac{\spr{\psi_{x,y}}{\partial_x H(x,y) \psi_{x,y}}}{\spr{\psi_{x,y}}{\partial_y H(x,y) \psi_{x,y}}}.
 \]
 The claim follows by some more computations.
\end{proof}

Theorem~\ref{thm:evperturb2} is not good enough for our purposes,
since we will need a better estimate on $\|\xi\|_{L^{\infty}([-s,s])}$.
We will prove

\begin{theorem}\label{thm:evperturb3}
 Let $\eps > \eta > 0$, $\delta \in (0,\frac{1}{3})$.
 Assume \eqref{eq:HxHyC}, \eqref{eq:E0etaepsevH00}, and
 \eqref{eq:sprdH00}, and define $s$ by \eqref{eq:desperev2}.
 Furthermore let $\mathfrak{X} \subseteq [-s,s]$ and assume that
 for $x \in \mathfrak{X}$, we have
 \be
  |\lambda(x,0) - E_0| \leq \frac{\delta}{3} \eta.
 \ee
 Then, we can choose a function $\xi$ such that $|\xi'(x)| \leq \delta$,
 $|\xi(x)| \leq \eta$, and for $x\in\mathfrak{X}$, we have
 \be
  E_0\text{ is a $\frac{\eps}{2}$-isolated eigenvalue of }H(x,\xi(x))
 \ee
\end{theorem}

Let $\xi_0$ be given by the previous theorem. A simple estimate
shows the claimed estimate whenever $x\in\mathfrak{X}$.

\begin{proof}[Proof of Theorem~\ref{thm:evperturb3}]
 Define $a = \inf(\mathfrak{X})$, $b = \sup(\mathfrak{X})$
 and introduce
 \[
  \xi(x) = \begin{cases} \xi_0(a),&x \leq a; \\
  \xi_0(x), & x \in [a,b];\\
  \xi_0(b), & x \geq b.\end{cases}
 \]
 The claims now follow by some computations.
\end{proof}

%
%
%

\section{Proof of Theorem~\ref{thm:extendpara}}

In order to prove Theorem~\ref{thm:extendpara},
we will need the following proposition, which allows us to
construct continuously differentiable functions $\T \to \T$.

\begin{proposition}\label{prop:C1circle}
 Given $L_0 > 0$, $\eps > 0$, $\delta > 0$, $\mathfrak{X}\subseteq\T$,
 and a continuously differentiable
 function $\hat{\xi}_0:\T\to\T$ satisfying for $x\in\T$
 \be
  |\hat{\xi}_0'(x)| \leq L_0.
 \ee
 Let furthermore $I_1, \dots, I_Q \subseteq \T$ be disjoint intervals
 satisfying $|I_q| \geq \eps$ and $\xi_q: I_q \to \T$ be 
 continuously differentiable functions satisfying for $x \in I_q$
 \be
  |\xi_q'(x)| \leq L_0,\quad |\xi_q(x) - \hat{\xi}_0(x)| < \delta.
 \ee
 Then there exists a subset $\mathcal{Q} \subseteq \{1,\dots, Q\}$ such that
 \be
  \left|\bigcup_{q\in\mathcal{Q}} I_q \cap \mathfrak{X}\right| \geq 
   \frac{1}{3} \left|\bigcup_{q=1}^{Q} I_q \cap \mathfrak{X}\right|
 \ee
 and a continuously differentiable function $\xi: \T\to\T$ 
 satisfying for $q \in \mathcal{Q}$
 that for $x \in I_q$ we have $\xi(x) = \xi_q(x)$ and the
 bound  
 \be
  \|\xi'\|_{L^{\infty}(\T)} \leq L_0 + 3 \frac{\delta}{\eps},
   \quad
    \|\xi - \hat{\xi}_0\|_{L^{\infty}(\T)}  \leq 5 \delta.
 \ee
\end{proposition}

We begin by constructing the set $\mathcal{Q}$. For
intervals $I, \ti{I} \subseteq \T$, we write
\be
 \dist(I,\ti{I}) = \inf_{x \in I,\ \ti{x}\in\ti{I}} \|x - \ti{x}\|.
\ee

\begin{lemma}
 There exists $\mathcal{Q}$ such that
 \begin{enumerate}
  \item For $q, \ti{q} \in \mathcal{Q}$ with $q \neq \ti{q}$, we have
   \be
    \dist(I_q, I_{\ti{q}}) \geq \eps.
   \ee
  \item We have
   \be
    \left|\bigcup_{q\in\mathcal{Q}} I_q\right| \geq 
     \frac{1}{3} \left|\bigcup_{q=1}^{Q} I_q\right|
   \ee  
 \end{enumerate}
\end{lemma}

\begin{proof}
 Write $I_q = [a_q, b_q]$ and order the intervals such that
 \[
  0 \leq a_1 < b_1 < a_2 < b_2 < \dots < a_Q < b_Q.
 \]
 We may have $b_Q > 0$. For $Q = 2 P$ even define the three
 sets
 \[
  \mathcal{Q}_1 = \{2p,\quad 1 \leq p \leq P\},\quad
  \mathcal{Q}_2 = \{2p-1,\quad 1 \leq p \leq P\},\quad
  \mathcal{Q}_3 = \emptyset.
 \]
 For $Q = 2 P + 1$ odd, define
 \[
  \mathcal{Q}_1 = \{2p,\quad 1 \leq p \leq P\},\quad
  \mathcal{Q}_2 = \{2p-1,\quad 1 \leq p \leq P\},\quad
  \mathcal{Q}_3 = \{Q\}.
 \]
 Since $|I_q| \geq \eps$ and we ordered the intervals,
 property (i) holds. Furthermore, since the sets
 $\mathcal{Q}_1, \mathcal{Q}_2, \mathcal{Q}_3$ are disjoint,
 there exists $j \in \{1,2,3\}$ such that property (ii)
 holds. Chose $\mathcal{Q} = \mathcal{Q}_j$, finishing the proof.
\end{proof}

Define $J = \bigcup_{q \in\mathcal{Q}} I_q$ and introduce
for $x \in J$
\be
 \eta_1(x) = \xi_q(x) - \hat{\xi}_0(x),\quad \text{if } x\in I_q.
\ee
By assumption, we now have that $\|\eta_1\|_{L^{\infty}(J)} \leq \eta$
and $\|\eta_1'\|_{L^{\infty}(J)} \leq L_0$. We now have that

\begin{lemma}
 There exists a continuously differentiable map $\eta:\T\to\T$ such that
 $\eta(x) = \eta_1(x)$ for $x \in J$ and $\|\eta'\|_{L^\infty(\T)} \leq L + 3 \frac{\delta}{\eps}$.
\end{lemma}

\begin{proof}
 We can write $J = \bigcup_{p=1}^{P} [a_{p}, b_{p}]$ with
 \[
  a_1 < b_1 < a_2 < b_2 < \dots < a_P < b_P.
 \]
 We then have that $|a_{p+1} - b_p| \geq \eps$ and that $|\eta_1(b_p) - \eta_1(a_{p+1})| \leq \delta$.
 The claim then follows by some computations.
\end{proof}

\begin{proof}[Proof of Proposition~\ref{prop:C1circle}]
 Define $\xi = \hat{\xi}_0 + \eta$. The claim follows by some
 more computations.
\end{proof}

Having established this preliminary proposition, we will
now proceed to prove Theorem~\ref{thm:extendpara}.
The first step is to apply Theorem~\ref{thm:evperturb2}.
Define $Q = \lceil \frac{50 R^2 \cdot \|f'\|_{L^{\infty}(\T)} ^{2}}{\eps d^2}\rceil$
and a sequence of intervals
\be
 I_q = \left[\frac{q-1}{Q}, \frac{q}{Q}\right)
\ee
for $q = 1, \dots, Q$. We have that the $I_q$
partition $\T$. Call $q$ {\em good}, if $I_q \cap \widetilde{\mathfrak{X}}\neq\emptyset$.

\begin{lemma}
 We have that
 \be
  \left|\bigcup_{q\text{ good}} I_q \cap \widetilde{\mathfrak{X}}\right| = |\widetilde{\mathfrak{X}}|.
 \ee
\end{lemma}

\begin{proof}
 It is easy to see that
 $\widetilde{\mathfrak{X}}\subseteq\bigcup_{q\text{ good}} I_q$,
 which implies the claim. 
\end{proof}

For each good $q$, we choose $x_q \in I_q \cap \widetilde{\mathfrak{X}}$
and define
\be
 H_{q}(x,y) = H_{h, x_q + x, \xi(x_q + x) + y}^{[-R,R]}.
\ee
We note that, we have that
\be
 \|\partial_x H_{q}(x,y)\| \leq R \cdot \|f'\|_{L^{\infty}(\T)},
 \quad
 \|\partial_y H_{q}(x,y)\| \leq \|f'\|_{L^{\infty}(\T)}.
\ee
We have

\begin{lemma}
 Let $d$ be defined as in \eqref{eq:defdfprime}.
 Then for $\psi$ the eigenfunction of $H(0,0)$ corresponding
 to the eigenvalue in the interval $[E_0 - \eps, E_0 + \eps]$, we have
 \be
  |\spr{\psi}{\partial_x H_{q} \psi}| \leq \frac{1}{2} d^5,\quad
  |\spr{\psi}{\partial_y H_{q} \psi}| \geq 2 d
 \ee
\end{lemma}

\begin{proof}
 By Lemma~\ref{lem:sprstable}, we have that
 \[
  \spr{\psi_1}{V_x \psi_1}| \leq \frac{1}{4} d^5,\quad |\spr{\psi_1}{V_y \psi_1}| \geq 4 d,
 \]
 with $\psi_1$ the eigenfunction of $H^{[-M,M]}_{h,x_q, \xi(x_q)}$.
 The claim now follows by the extension property from $H^{[-M,M]}$
 to $H^{[-R,R]}$.
\end{proof}

We thus have that the conditions of Theorem~\ref{thm:evperturb1}
and Theorem~\ref{thm:evperturb2} hold. Using these, we obtain

\begin{lemma}
 Let $q$ be good. Then there exists a curve
 $\xi_q: I_q \to \T$ such that for $x \in I_q \cap \widetilde{\mathfrak{X}}$,
 we have
 \be\label{eq:E0evofHxiq}
  E_0\text{ $2 \eta$-extends to an $\frac{\eps}{2}$-isolated eigenvalue of }
   H^{[-R,R]}_{h,x,\xi_q(x)}.
 \ee
 Furthermore, $\xi_q$ satisfies
 \begin{enumerate}
  \item For $x \in I_q$, we have $|\xi_q'(x)| \leq \frac{1}{10}$.
  \item For $x \in I_q \cap \mathfrak{X}$, we have
   \be
    |\xi_q(x) - \xi(x)| \leq \eta.
   \ee
  \item For $x \in \partial(I_q)$, we have $|\xi_q(x) - \xi(x)| \leq \eta$.
 \end{enumerate}
\end{lemma}

\begin{proof}
 By Theorem~\ref{thm:evperturb3}, we obtain for each good $q$
 a map $\ti{\xi}_q$ defined on the interval $[-\frac{1}{Q}, \frac{1}{Q}]$
 such that
 \[
  E_0\text{ is a $\frac{\eps}{2}$-isolated eigenvalue of } H(x, \ti{\xi}_q(x)).
 \]
 Define $\xi_q(x) = \ti{\xi}_q(x - x_q) + \xi(x)$. It is easy to check
 that the previous equation implies for $x \in I_q$ that
 \[
  E_0\text{ is a $\frac{\eps}{2}$-isolated eigenvalue of } H_{h,x, \ti{\xi}_q(x)}^{[-R,R]}.
 \]
 Furthermore, we have that $|\xi_q(x) - \xi(x)| \leq \eta$.,
 We will now check \eqref{eq:E0evofHxiq}. 
 First, one can describe
 the conclusions in the furthermore of Theorem~\ref{thm:evperturb1}
 as
 \[
  E_0\text{ extends to a $\frac{\eps}{2}$-isolated eigenvalue of } H(x, \ti{\xi}_q(x))
  \text{ from } H(x, y)
 \]
 for any $y$. This combined with assumption of Theorem~\ref{thm:extendpara}
 implies \eqref{eq:E0evofHxiq}.
\end{proof}

\begin{proof}[Proof of Theorem~\ref{thm:extendpara}]
 Proposition~\ref{prop:C1circle}  allows us to construct a
 function $\hat{\xi}$ such that for $x\in\T$
 \[
  |\hat{\xi}'(x)| \leq L + \eps,\quad |\hat{\xi}(x) - \xi(x)| \leq 2 \eta
 \]
 and for $q \in \mathcal{Q}$, $x \in \mathfrak{X} \cap I_q$,
 we have the claimed properties of the eigenvalue $E_0$ and
 the eigenfunctions.
\end{proof}

%
%
%


%
%
%

\section*{Acknowledgements}

I started working on the problem under consideration while
writing my PhD thesis. I am thankful to my advisor David Damanik
for the support during this process. Furthermore, I wish to
thank Michael Goldstein and Svetlana Jitomirskaya for insightful
discussions to the process of creation of gaps. Last, I wish
to thank Mark Embree without whose help, I would probably not
been able to produce the numerical results in Appendix~\ref{sec:numeric}.

%
%
%

\appendix

%
%
%

\section{Numerical evidence}
\label{sec:numeric}

The essential problem in computing the spectrum of a Schr\"odinger operator $H$
acting on $\ell^2(\Z)$ is that this is an infinite dimensional space, so
we will need to approximate the spectrum of $H$ by $H^{[-N,N]}$ for some 
large $N$. The essential insight is that the difference
\be
 H - H^{(-\infty, -N-1]} \oplus H^{[-N,N]} \oplus H^{[N+1, \infty)}
\ee
is a rank $4$ operator. Hence, if an interval $[a,b]$ contains more
than $5$ eigenvalues of $H^{[-N,N]}$ then
\be
 \sigma(H) \cap [a,b] \neq \emptyset.
\ee
The conclusion of this is that, if we denote by
\be
 E_{-N}^{[-N,N]} < E_{-N+1}^{[-N,N]} < \dots < E_{N-1}^{[-N,N]} < E_{N}^{[-N,N]}
\ee
the eigenvalues of $H^{[-N,N]}$ and define
\be
 \delta^{[-N,N]} = \inf_{-N \leq j \leq N - 5} \left( E_{j + 5}^{[-N,N]} - E_{j}^{[-N,N]} \right),
\ee
then we have that the spectrum of $H$ is at least $\delta^{[-N,N]}$-dense
in $[E_{-N+4}^{[-N,N]}, E_{N-4}^{[-N,N]}]$.

In the following table, I show the results for
\be
 V(n) = 2 \cos(2\pi \sqrt{2} n^2)
\ee
for $h=1$. It should be noted that this computation takes
approximately 10 minutes, and that since the density decreases
with a similar rate as $N$ grows, one should expect the spectrum
of $\Delta + V$ to be an interval.

\begin{figure}[ht]
\begin{tabular}{l|l}
N & density of the spectrum \\ \hline
320	 &	 0.291089\\
640	 &	 0.231054\\
1280	 &	 0.174700\\
2560	 &	 0.139430\\
5120	 &	 0.063408\\
10240	 &	 0.025548\\
20480	 &	 0.013934\\
40960	 &	 0.009152\\
\end{tabular}
\caption{$h=1$}
\end{figure}

%
%
%

\section{The top and bottom of the spectrum}
\label{sec:topbotspec}

In this appendix, I show how to obtain bounds at the top $E_+$
of the spectrum. It is easy to see that a similar bound is
valid for the bottom of the spectrum. We have that
\be
 E_+(h) = \sup(\sigma(H_{h,x,y})) = \sup_{\|\varphi\| = 1} \spr{\varphi}{H_{h,x,y} \varphi}
\ee
where $H_{h,x,y}$ was defined in \eqref{eq:defHhxy}
and the second inequality follows from the minimax principle.
We have that
\begin{align}
\nn \spr{\varphi}{H_{h,x,y} \varphi} &= \sum_{n \in \Z} \left(
  h \cdot \varphi(n) \Big(\varphi(n+1) + \varphi(n-1)\Big)
  + V_{x,y}(n) \varphi(n)^2\right), \\
 V_{x,y}(n) &= f(y + n x + n (n-1) \alpha).
\end{align}
It is easy to see that $E_+(h) \leq \max(f) + 2 h$.
Choose now $y_0$ such that $f(y_0) = \max(f)$ and
consider
\be
 \varphi_0(n) = \begin{cases} \frac{1}{\sqrt{2}}, &n \in \{0,1\};\\
 0, &\text{otherwise}. \end{cases}
\ee
A computation shows that
\be
 \|\varphi_0\| = 1,\quad \spr{\varphi_0}{H_{h,0, y_0} \varphi_0} = f(y_0) + h.
\ee
Hence, we have obtained that
\be
 \max(f) + h \leq E_+(h) \leq \max(f) + 2 h,
\ee
which is all we claimed.

%
%
%

\end{document}